\documentclass[final,leqno,onefignum,onetabnum]{siamltex}
\usepackage{fullpage}
\usepackage{upgreek}
\usepackage[mathscr]{eucal}
\usepackage{color}
\usepackage{amsmath,amssymb,amsopn,mathtools}
\usepackage{graphicx}

\usepackage{hyperref}       % internal hyperlink to contents
\hypersetup{
	colorlinks=true,   % boxed links
	citecolor=blue,     % citation links (bibliography)
	filecolor=blue,     % color of file links
	linkcolor=red,    % internal links (sections, pages, etc.)
	urlcolor=blue}      % color of URL links (mail, web)

\usepackage{relsize}
\numberwithin{equation}{section}
\usepackage{enumitem}

\newlist{Assumption}{enumerate}{1}
\setlist[Assumption]{label=A\arabic*}

\definecolor{Blue}{rgb}{0,0,1}
\definecolor{Red}{rgb}{1,0,0}
\definecolor{Green}{rgb}{0,1,0}
\newcommand{\reviewerA}[1]{#1}
\newcommand{\reviewerB}[1]{#1}
\newcommand{\reviewerBRtwo}[1]{#1}
\newcommand{\reviewerBRthree}[1]{#1}
\newcommand{\ourReReading}[1]{#1}
\newcommand{\fontContinuousFunc}{ }
\newcommand{\fontDiscreteFunc}{\mathscr}
\newcommand{\fontDiscrete}{\mathcal}
\newcommand{\knArg}[1]{k(t^{#1})}
\newcommand{\kn}{\knArg{n}}
\newcommand{\knDummy}{k(\timeDummy^n)}
\newcommand{\kmax}{k_\text{max}}
\newcommand{\bds}[1]{{\boldsymbol{#1}}}

\newcommand{\defeq}{:=}

\newcommand{\SPSD}[1]{\text{SPSD}(#1)}
\newcommand{\optDummy}{v}
\newcommand{\orthomat}[2]{\mathbb V_{{#2}}(\RR{{#1}})}

\newcommand{\mappingNo}{\mathcal I}
\newcommand{\mapping}[2]{\mappingNo(#1,#2)}

\newcommand{\methodAcronym}{ST-LSPG}
\newcommand{\timeSpace}{\fontContinuousFunc H}
\newcommand{\timeDiscreteFuncSpace}{\fontDiscreteFunc H}
\newcommand{\timeDiscreteSpace}{\RR{\ntimedof}}
\newcommand{\spatialSpace}{\RR{\nspacedof}}
\newcommand{\unrollfuncNo}{\bds{g}}
\newcommand{\unrollfunc}[1]{{\unrollfuncNo}(#1)}
\newcommand{\vectorizeNone}{\vectorizefuncNo} % timestep at optional
\newcommand{\vectorize}[1]{\ensuremath{\vectorizeNone(#1)}}  % timestep
\newcommand{\vectorizeinv}[1]{\ensuremath{\vectorizeNone^{-1}(#1)}}  %
\newcommand{\vectorizefuncNo}{\bds{h}}

\newcommand{\lipschitzK}{L_{\res}}
\newcommand{\inverseLipschitzK}{K_{\res}}
\newcommand{\lipschitzFK}{L_{\flux}}

\newcommand{\stateContinuousDummyEntryNo}{\boldsymbol{u}}

\newcommand{\stateContinuousDummy}[1]{\bds{\stateContinuousDummyEntryNo}} 
\newcommand{\stateTwoContinuousDummyEntryNo}{\boldsymbol{v}}

\newcommand{\stateTwoContinuousDummy}[1]{\bds{\stateTwoContinuousDummyEntryNo}} 
\newcommand{\metric}{\bds{\Theta}} 

\newcommand{\crossareaSymb}{A}
\newcommand{\externalforceSymb}{q}
\newcommand{\externalforce}{\boldsymbol \externalforceSymb}

\newcommand{\ndof}{\nspacedof}

\newcommand{\weightmatRows}{z}

\newcommand{\weightmatstRows}{\bar z}
\newcommand{\weightmatstRowsArg}[1]{\weightmatstRows^{#1}}
\newcommand{\weightmatstRowsn}{\weightmatstRowsArg{n}}

\newcommand{\unitvec}{\bds{e}}
\newcommand{\unitvecArg}[1]{\unitvec_{#1}}
\newcommand{\spacetimeindexNo}{\varphi}
\newcommand{\spacetimeindex}[2]{\varphi(#1,#2)}
\newcommand{\spacetimeindexIncludeAll}[1]{\spacetimeindex{\spaceindex{#1}}{\timeindex{#1}}}

\newcommand{\spaceindexNo}{\mathcal s}
\newcommand{\timeindexNo}{\mathcal t}
\newcommand{\spaceindex}[1]{\mathcal s_{#1}}
\newcommand{\timeindex}[1]{\mathcal t_{#1}}

\newcommand{\spacetimesampleset}{\mathfrak{st}}

\newcommand{\range}[1]{\text{range}(#1)}

\newcommand{\ndofPorts}[1]{n^p}

\newcommand{\card}[1]{|#1|}

\newcommand{\sampleSpaceGenCoord}{\mathcal D_{\redsolapproxST}}

\newcommand{\sigmamax}[1]{\sigma_\mathrm{max}(#1)}
\newcommand{\sigmamin}[1]{\sigma_\mathrm{min}(#1)}
\newcommand{\normEquiv}{P}
\newcommand{\lebesgue}{\Lambda}

\newcounter{remctr}
\newcounter{propctr}
\newcounter{proposctr}
\newcommand{\RR}[1]{\ensuremath{\mathbb{R}^{ #1 }}}
\newcommand{\NN}{\mathbb{N}}
\newcommand{\RRplus}[1]{\ensuremath{\mathbb{R}_+^{ #1 }}}
\newcommand{\natNo}{\NN}
\newcommand{\nat}[1]{\natNo(#1)}
\newcommand{\natZero}[1]{\{0,\ldots,#1\}}
\newcommand{\innat}[1]{\in\nat{#1}}
\newcommand{\innatSeq}[1]{=1,\ldots,#1}
\newcommand{\innatZero}[1]{\in\natZero{#1}}

\newcommand{\VV}{\mathbb{V}}

\newcommand{\Span}[1]{\mathrm{span}\{#1\}}

\newcommand{\identity}[1]{\boldsymbol I_{#1}}
\newcommand{\nstepExtraction}{P_n}
\newcommand{\lipschitz}{\kappa}
\newcommand{\lipschitzst}{\bar{\lipschitz}}

\newcommand{\pexit}{P_\text{exit}}
\newcommand{\specificheat}{\gamma}
\newcommand{\energypermass}{\epsilon}
\newcommand{\densitySymb}{\rho}
\newcommand{\velocitySymb}{u}
\newcommand{\energydensity}{e}
\newcommand{\pressureSymb}{p}
\newcommand{\specificgasconstant}{R}
\newcommand{\speedofsoundSymb}{c}
\newcommand{\machSymb}{M}
\newcommand{\temperatureSymb}{T}
\newcommand{\resSymb}{r}
\newcommand{\resRedSymb}{{\mathsf r}}
\newcommand{\fluxSymb}{f}
\newcommand{\paramSymb}{\mu}
\newcommand{\solSymb}{x}
\newcommand{\tensor}{\mathcal X}
\newcommand{\tensorUnfold}[1]{\boldsymbol{X}_{(#1)}}
\newcommand{\tensorReduceNo}{\mathcal X}
\newcommand{\tensorReduce}[1]{\tensorReduceNo(#1)}
\newcommand{\tensorresReduce}[1]{\tensorresReduceNo(#1)}
\newcommand{\tensorReduceUnfold}[2]{\boldsymbol{X}(#1)_{(#2)}}
\newcommand{\tensorres}{\mathcal R}
\newcommand{\tensorresUnfold}[1]{\boldsymbol{R}_{(#1)}}
\newcommand{\tensorresReduceNo}{\mathcal R}
\newcommand{\tensorresReduceUnfold}[2]{\boldsymbol{R}(#1)_{(#2)}}

\newcommand{\spaceSymb}{s}
\newcommand{\timeSymb}{t}

\newcommand{\singvalmatSymb}{\Sigma}
\newcommand{\leftsingmatSymb}{U}
\newcommand{\rightsingmatSymb}{V}
\newcommand{\leftsingvecSymb}{u}

\newcommand{\weightmatSymb}{A}
\newcommand{\weightvecSymb}{a}
\newcommand{\basismatspaceSymb}{\Phi}
\newcommand{\basisvecspaceSymb}{\phi}

\newcommand{\basisvectimeSymb}{\psi}
\newcommand{\basisvectimeFuncSymb}{\uppsi}
\newcommand{\basismatstSymb}{\Upsilon}

\newcommand{\samplematSymb}{Z}
\newcommand{\samplevecSymb}{z}
\newcommand{\dummySymb}{y}
\newcommand{\snapshotSymb}{X}

\newcommand{\stabilityconstantSymb}{h}

\newcommand{\nparam}{n_\mu}
\newcommand{\ntrain}{n_\text{train}}

\newcommand{\nrestrain}{n_\text{res}}
\newcommand{\nressample}{n_{\samplevecSymb}}
\newcommand{\nressamplest}{\bar n_{\samplevecSymb}}

\newcommand{\nrestimeind}{\bar n_{\timeSymb}}
\newcommand{\nresspaceind}{\bar n_{\spaceSymb}}
\newcommand{\nbasisres}{n_\resSymb}
\newcommand{\nbasisresst}{\bar n_\resSymb}
\newcommand{\nbasisspace}{{n_\spaceSymb}}
\newcommand{\nbasisspaceres}{{n_{\resSymb,\spaceSymb}}}
\newcommand{\nbasistime}{n_\timeSymb}
\newcommand{\nbasistimeres}{n_{\resSymb,\timeSymb}}
\newcommand{\nbasisst}{{n_{\spaceSymb\timeSymb}}}

\newcommand{\nspacedof}{N_\spaceSymb}
\newcommand{\ntimedof}{{N_\timeSymb}}
\newcommand{\nTimestepsFine}{\ntimedof}

\newcommand{\paramDomain}{\mathcal D}
\newcommand{\paramDomainTrain}{\paramDomain_\text{train}}
\newcommand{\paramDomainTrainres}{\paramDomain_\text{res}}
\newcommand{\ones}{\onebold}
\newcommand{\onesFunc}{\fontDiscreteFunc O}
\newcommand{\totaltime}{T}
\newcommand{\spatialSubspace}{\fontDiscrete S}
\newcommand{\temporalSubspaceFunc}{\fontDiscreteFunc T}

\newcommand{\spacesampleset}{\mathfrak{s}}
\newcommand{\timesampleset}{\mathfrak{t}}
\newcommand{\PG}{\text{PG}}
\newcommand{\Tbe}{\boldsymbol A_\text{BE}}
\newcommand{\Alm}{\boldsymbol A_\text{LM}}
\newcommand{\Blm}{\boldsymbol B_\text{LM}}
\newcommand{\objective}{\zeta}
\newcommand{\operatorerrorboundSymb}{\Lambda}

\newcommand{\stabilityconstantBE}{\bar{\stabilityconstantSymb}_\text{be}}

\newcommand{\resSymbst}{\bar{\resSymb}}
\newcommand{\resRedSymbst}{\bar{\resRedSymb}}
\newcommand{\res}{\boldsymbol \resSymb}

\newcommand{\resArg}[1]{\boldsymbol \resSymb^{#1}}
\newcommand{\resArgEntry}[2]{\resSymb_{#1}^{#2}}
\newcommand{\resRed}{\boldsymbol \resRedSymb}
\newcommand{\resRedApprox}{\tilde{\boldsymbol \resRedSymb}}

\newcommand{\resn}{\resArg{n}}
\newcommand{\resst}{\boldsymbol \resSymbst}
\newcommand{\resstApprox}{\tilde {\res}}
\newcommand{\resRedst}{\boldsymbol \resRedSymbst}

\newcommand{\param}{\boldsymbol \paramSymb}
\newcommand{\paramTrain}[1]{\param^{#1}_\text{train}}
\newcommand{\paramTrainres}[1]{\param^{#1}_\text{res}}
\newcommand{\paramDummy}{\boldsymbol \nu}
\newcommand{\dt}{\Delta \timeSymb}
\newcommand{\dtArg}[1]{\reviewerA{\dt^{#1}}}
\newcommand{\sol}{\boldsymbol \solSymb}
\newcommand{\solw}{\boldsymbol w}
\newcommand{\solTimeContinuous}{\sol^\star}
\newcommand{\solst}{\bar{\sol}}
\newcommand{\solSymbst}{\bar{\solSymb}}
\newcommand{\flux}{\boldsymbol \fluxSymb}

\newcommand{\rhsFunc}{\boldsymbol b}
\newcommand{\rhsFuncArg}[1]{\rhsFunc(t^{#1})}
\newcommand{\fluxst}{\bar{\flux}}

\newcommand{\solpgst}{\solst_{PG}}

\newcommand{\solDummy}{\boldsymbol w}

\newcommand{\solDummyFuncArg}[1]{\solDummy(t^{#1})}
\newcommand{\solDummyFuncArgDummy}[1]{\solDummy(\timeDummy^{#1})}

\newcommand{\solRedDummyScalar}{ \hat w}
\newcommand{\solRedDummyScalarArg}[1]{\solRedDummyScalar_{#1}}
\newcommand{\solRedDummy}{\hat {\boldsymbol w}}
\newcommand{\solDummyOpt}{{\boldsymbol \optDummy}}
\newcommand{\solRedDummyOpt}{\hat {\boldsymbol \optDummy}}
\newcommand{\timeArg}[1]{t^{#1}}
\newcommand{\timeDummy}{\tau}
\newcommand{\timeDomain}{[0,\totaltime]}
\newcommand{\timestepset}{\{\timeArg{n}\}_{n=0}^{\ntimedof}}
\newcommand{\solFunc}{\sol}
\newcommand{\solinfty}{\sol_\infty}
\newcommand{\soltwo}{\sol_2}
\newcommand{\solFuncArg}[1]{\solFunc(t^{#1};\param)}
\newcommand{\solFuncOnlyParam}{\solFunc(\cdot;\param)}
\newcommand{\solFuncOnlyParamTrain}[1]{\solFunc(\cdot;\paramTrain{#1})}
\newcommand{\solFuncOnlyParamProject}{\hat\solFunc(\param)}

\newcommand{\solArg}[1]{\sol(t^{#1};\param)}
\newcommand{\soln}{\solArg{n}}
\newcommand{\solnm}{\solArg{n-1}}
\newcommand{\spacetimesubspaceFunc}{\fontDiscreteFunc{ST}}
\newcommand{\spacetimesubspace}{\fontDiscrete{ST}}

\newcommand{\solapproxFunc}{\solapprox}
\newcommand{\solapproxFuncArg}[1]{\solapproxFunc(t^{#1};\param)}

\newcommand{\solapproxFuncOnlyParam}{\solapproxFunc(\cdot;\param)}

\newcommand{\redsolapproxFunc}{\redsolapprox}

\newcommand{\redsolapproxFuncOnlyParam}{\redsolapproxFunc(\cdot;\param)}

\newcommand{\redsolapproxEntryFunc}{\hat x}
\newcommand{\redsolapproxEntryFuncArg}[2]{\redsolapproxEntryFunc_{#1}(t^{#2};\param)}
\newcommand{\redsolapproxEntryFuncArgGen}[2]{\redsolapproxEntryFunc_{#1}(#2;\param)}
\newcommand{\redsolapproxEntryFuncOnlyParam}[1]{\redsolapproxEntryFunc_{#1}(\cdot;\param)}

\newcommand{\lineSearchParam}{\alpha}
\newcommand{\lineSearchParamArg}[1]{\lineSearchParam^{(#1)}}

\newcommand{\testbasisArg}[2]{\boldsymbol \xi_{ij}}

\newcommand{\solapproxArg}[1]{\solapprox(t^{#1};\param)}

\newcommand{\redsolapproxArg}[1]{\redsolapprox(t^{#1};\param)}

\newcommand{\solSTscalar}{y}
\newcommand{\solST}{\boldsymbol{\solSTscalar}}
\newcommand{\solapproxSTFunc}{\tilde\solST}
\newcommand{\solapproxST}{\tilde\solST}
\newcommand{\solapproxSTFuncArg}[1]{\solapproxSTFunc(t^{#1};\param)}

\newcommand{\solapproxSTFuncOnlyParam}{\solapproxSTFunc(\cdot;\param)}

\newcommand{\redsolapproxST}{\hat\solST}

\newcommand{\redsolapproxSTArg}[1]{\hat\solSTscalar_{#1}(\param)}
\newcommand{\redsolapproxSTArgNone}[1]{\hat\solSTscalar_{#1}}
\newcommand{\redsolapproxSTIt}[1]{\hat\solST^{(#1)}}
\newcommand{\deltaredsolapproxSTIt}[1]{\delta\hat\solST^{({#1})}}
\newcommand{\redsolapproxTrainNo}{\hat\solST_\text{res}}
\newcommand{\redsolapproxTrainArg}[1]{\redsolapproxTrainNo^{#1}}

\newcommand{\stbasismat}{\boldsymbol \Pi}
\newcommand{\stbasisvec}[1]{\boldsymbol \pi_{#1}}
\newcommand{\stbasisvecFunc}[1]{\boldsymbol \uppi_{#1}}
\newcommand{\mseError}{\text{relative error}}

\newcommand{\solnpg}{\sol_{PG}^n}

\newcommand{\solapprox}{\tilde\sol}
\newcommand{\redsolapprox}{\hat\sol}
\newcommand{\solapproxPG}{\solapprox_\PG}
\newcommand{\redsol}{\hat\sol}

\newcommand{\redsolst}{\redsol}

\newcommand{\solEntry}[1]{\solSymb_{#1}}
\newcommand{\resEntryi}[1]{\resSymb_{#1}}
\newcommand{\resRedEntryi}[1]{\mathsf\resSymb_{#1}}

\newcommand{\dummy}{\boldsymbol {\dummySymb}}
\newcommand{\reddummy}{\hat{\dummy}}

\newcommand{\solinit}{\sol^0}
\newcommand{\solinitEntry}[1]{\solSymb_{#1}^0}

\newcommand{\solref}{\solinit(\param)}
\newcommand{\solrefst}{\bar{\sol}_{\text{ref}}}
\newcommand{\singvalmat}{\boldsymbol \singvalmatSymb}
\newcommand{\singvalmatspace}{\singvalmat_\spaceSymb}
\newcommand{\singvalmatspaceres}{\singvalmat_{\resSymb,\spaceSymb}}
\newcommand{\singvalmattime}{\singvalmat_\timeSymb}
\newcommand{\singvalmattimeres}{\singvalmat_{\resSymb,\timeSymb}}
\newcommand{\leftsingvec}{\boldsymbol{\leftsingvecSymb}}

\newcommand{\leftsingvecspace}{\leftsingvec_\spaceSymb}
\newcommand{\leftsingvecspaceres}{\leftsingvec_{\resSymb,\spaceSymb}}
\newcommand{\leftsingvectime}{\leftsingvec_\timeSymb}
\newcommand{\leftsingvectimeres}{\leftsingvec_{\resSymb,\timeSymb}}
\newcommand{\leftsingvectimeArg}[1]{\leftsingvectime^{#1}}
\newcommand{\leftsingvectimeresArg}[1]{\leftsingvectimeres^{#1}}

\newcommand{\leftsingvecspacei}[1]{\leftsingvecspace^{#1}}
\newcommand{\leftsingvecspaceresi}[1]{\leftsingvecspaceres^{#1}}

\newcommand{\leftsingmat}{\boldsymbol{\leftsingmatSymb}}
\newcommand{\rightsingmat}{\boldsymbol{\rightsingmatSymb}}
\newcommand{\leftsingmatspace}{\leftsingmat_\spaceSymb}
\newcommand{\leftsingmatspaceres}{\leftsingmat_{\resSymb,\spaceSymb}}
\newcommand{\leftsingmattime}{\leftsingmat_\timeSymb}

\newcommand{\leftsingmattimeres}{\leftsingmat_{\resSymb,\timeSymb}}

\newcommand{\rightsingmatspace}{\rightsingmat_\spaceSymb}
\newcommand{\rightsingmatspaceres}{\rightsingmat_{\resSymb,\spaceSymb}}
\newcommand{\rightsingmattime}{\rightsingmat_\timeSymb}
\newcommand{\nnzNo}{\mathrm{nnz}}
\newcommand{\nnz}[1]{\nnzNo(#1)}
\newcommand{\rightsingmattimeres}{\rightsingmat_{\resSymb,\timeSymb}}

\newcommand{\basisvectime}{\boldsymbol{\basisvectimeSymb}}
\newcommand{\basisvectimei}[1]{\basisvectime_{#1}}
\newcommand{\basisvectimeij}[2]{\basisvectime^{#1}_{#2}}

\newcommand{\basisvectimeresi}[1]{\basisvectime_{r,#1}}
\newcommand{\basisvectimeresij}[2]{\basisvectime^{#1}_{r,#2}}

\newcommand{\basisvectimeFunc}{\boldsymbol{\basisvectimeFuncSymb}}
\newcommand{\basisvectimeFuncij}[2]{\basisvectimeFunc_{#2}^{#1}}
\newcommand{\basisvectimeresFuncij}[2]{\basisvectimeFunc_{r,#2}^{#1}}
\newcommand{\basisvectimeFunci}[1]{\basisvectimeFunc_{#1}}
\newcommand{\basisvectimeresFunci}[1]{\basisvectimeFunc_{r,#1}}

\newcommand{\operatorerrorbound}{\boldsymbol{\operatorerrorboundSymb}}
\newcommand{\basisDummy}{\boldsymbol{V}}
\newcommand{\basismatspace}{\boldsymbol{\basismatspaceSymb}}
\newcommand{\basismatspaceres}{\boldsymbol{\basismatspaceSymb}_r}
\newcommand{\basisvecspace}{\boldsymbol{\basisvecspaceSymb}}

\newcommand{\basisvecspacei}[1]{\boldsymbol{\basisvecspaceSymb}_{#1}}
\newcommand{\basisvecspaceresi}[1]{\boldsymbol{\basisvecspaceSymb}_{r,#1}}
\newcommand{\basismatres}{\basismatspace_\resSymb}
\newcommand{\basismatrest}{{\bar \basismatspace}_\resSymb}
\newcommand{\basismatrestvec}[1]{{\bar {\boldsymbol{\basisvecspaceSymb}}_{r,#1}}}

\newcommand{\basisres}{\boldsymbol \Pi_r}
\newcommand{\mappingresNo}{\mathcal I_r}
\newcommand{\mappingres}[2]{\mappingresNo(#1,#2)}
\newcommand{\basisresvecArg}[1]{\boldsymbol \pi_{r,#1}}
\newcommand{\basisresvecij}[2]{\basisresvecArg{\mappingres{#1}{#2}}}
\newcommand{\rfactor}{{\boldsymbol R}}

\newcommand{\basisvecressti}[1]{\basismatrestvec{#1}}

\newcommand{\errorvec}{\boldsymbol \varepsilon}
\newcommand{\errorvecFunc}{\boldsymbol \upvarepsilon}
\newcommand{\errorvecFuncEntry}[1]{\upvarepsilon_{#1}}
\newcommand{\timesamplestoadd}[1]{\nrestimeind^{#1}}
\newcommand{\spacesamplestoadd}[1]{\nresspaceind^{#1}}
\newcommand{\spacetimesamplestoadd}[1]{\nressamplest^{#1}}

\newcommand{\weightmat}{\boldsymbol{\weightmatSymb}}
\newcommand{\weightmatst}{\bar\weightmat}
\newcommand{\weightmatstArg}[1]{\weightmatst^{#1}}
\newcommand{\weightvecstij}[2]{\bar{\weightvecSymb}_{{#1}{#2}}}
\newcommand{\snapshots}{\boldsymbol \snapshotSymb}

\newcommand{\snapshotmatspace}{\tensorUnfold{1}}
\newcommand{\snapshotmattime}{\tensorUnfold{2}}

\newcommand{\samplemat}{\boldsymbol \samplematSymb}
\newcommand{\samplematst}{\bar{\samplemat}}

\newcommand{\spatialSubspacei}[1]{\spatialSubspace_{#1}}
\newcommand{\temporalSubspaceFuncArg}[1]{\temporalSubspaceFunc_{#1}}

\newcommand{\bmat}[1]{\begin{bmatrix}#1\end{bmatrix}} % Need \usepackage{amsmath}
\def\onebold{\boldsymbol{1}}

\def\wbold{\boldsymbol{w}}

\def\ybold{\boldsymbol{y}}

\def\zerobold{{\bf 0}}
\newtheorem{remark}{Remark}[section]
\usepackage{subfigure, algorithmic, algorithm}% subcaptions for subfigures
\usepackage{amsmath,amssymb}% subcaptions for subfigures
\title{Space--time least-squares Petrov--Galerkin projection for\\ nonlinear model reduction}
\author{
Youngsoo Choi\thanks{\ourReReading{Work was performed while employed in the Extreme-scale Data Science and Analytics
Department, Sandia National Laboratories, Livermore, CA 94550. Current
affiliation: Lawrence Livermore National Laboratory 
(\href{mailto:choi15@llnl.gov}{choi15@llnl.gov}).}}\and
Kevin Carlberg\thanks{Extreme-scale Data Science and Analytics
Department, Sandia National Laboratories, Livermore, CA 94550 
(\href{mailto:ktcarlb@sandia.gov}{ktcarlb@sandia.gov}).}}

\usepackage{epsfig, algorithmic, algorithm}
\usepackage{multirow, booktabs}
\usepackage{siunitx}
\usepackage{xr}
\mathtoolsset{showonlyrefs=true}
\usepackage{xr}

\title{Space--time least-squares Petrov--Galerkin projection for\\ nonlinear model reduction}
\author{
Youngsoo Choi\thanks{\ourReReading{Work was performed while employed in the Extreme-scale Data Science and Analytics
Department, Sandia National Laboratories, Livermore, CA 94550. Current
affiliation: Lawrence Livermore National Laboratory 
(\href{mailto:choi15@llnl.gov}{choi15@llnl.gov}).}}\and
Kevin Carlberg\thanks{Extreme-scale Data Science and Analytics
Department, Sandia National Laboratories, Livermore, CA 94550 
(\href{mailto:ktcarlb@sandia.gov}{ktcarlb@sandia.gov}).}}

\begin{document}
\setlength{\abovedisplayskip}{3pt}
\setlength{\belowdisplayskip}{3pt} 
\setlength{\abovedisplayshortskip}{3pt} 
\setlength{\belowdisplayshortskip}{3pt}

\maketitle

\begin{abstract}
This work proposes a space--time least-squares Petrov--Galerkin
(\methodAcronym) projection method for model reduction of nonlinear dynamical
systems.  In contrast to typical nonlinear model-reduction methods that first
apply (Petrov--)Galerkin projection in the spatial dimension and subsequently
apply time integration to numerically resolve the resulting low-dimensional
dynamical system, the proposed method applies projection in space and time
simultaneously. To accomplish this, the method first introduces a
low-dimensional space--time trial subspace, which can be obtained by computing tensor
decompositions of state-snapshot data. The method then computes
discrete-optimal approximations in this space--time trial subspace by
minimizing the residual arising after time discretization over all space and
time in a weighted $\ell^2$-norm. This norm can be defined
to enable complexity reduction (i.e., hyper-reduction) in time, which leads to
space--time collocation and space--time Gauss--Newton with Approximated
	Tensors (GNAT) variants of the \methodAcronym\
method. Advantages of the approach relative to typical
spatial-projection-based nonlinear model reduction methods such as Galerkin
projection and least-squares Petrov--Galerkin projection include a
reduction of both the spatial and temporal dimensions of the dynamical system,
%(2) the \reviewerBRtwo{possibility of removing} spurious temporal modes (e.g., unstable growth) from the
%state space, 
and \reviewerBRthree{\textit{a priori}} error bounds that
\reviewerBRthree{bound the solution error by the best space--time
approximation error and whose stability constants} exhibit slower growth in
time.
Numerical examples performed on model problems in fluid dynamics demonstrate
the ability of the method to generate orders-of-magnitude computational
savings relative to spatial-projection-based reduced-order models without
sacrificing accuracy \reviewerBRtwo{for a fixed spatio-temporal discretization}.
\end{abstract}
\begin{keywords}
 space--time projection, least-squares Petrov--Galerkin projection,
 residual minimization, model reduction, nonlinear dynamical systems
	\end{keywords}
	\begin{AMS}
	65L05,65L06,65L60,65M15,65M22,68U20
	\end{AMS}

\section{Introduction}

Reduced-order models (ROMs) of nonlinear dynamical systems are essential for
enabling high-fidelity computational models to be used in many-query and
real-time applications such as uncertainty quantification, design
optimization, and control.  Such ROMs reduce the \textit{spatial dimensionality} of the
dynamical system by performing a projection process on the governing system of
nonlinear ordinary differential equations (ODEs). The resulting ROM is then
resolved in time via numerical integration, typically with the same time
integrator and time step
employed for the high-fidelity model.
Unfortunately, many applications require simulating the model over long time
intervals, leading to high \textit{temporal dimensionality} characterized by
the number of time instances in the time discretization.  For example,
many	applications in fluid dynamics require long-time simulations to compute adequate statistics such as power spectral densities;
structural-dynamics applications can demand long-time integration when
structures undergo significant deformations; long-time integration is
necessary to assess stability of planetary orbits \cite{ito2002long};
molecular dynamics simulations \cite{elber2016perspective}
and condensed phase dynamics \cite{space1993long} also require long-time
integration.

As such, ROMs are often characterized by low spatial dimensionality, but high
temporal dimensionality, which can limit realizable computational savings in
practice.  It also renders ROMs ineffective in applications
that demand a low temporal dimension for computational tractability.  For
example, a high temporal dimension can render intrusive uncertainty
quantification methods (e.g., stochastic Galerkin
\cite{babuska2004galerkin,ghanem2003stochastic}) and simultaneous analysis and
design (SAND) in PDE-constrained optimization
\cite{haftka1989simultaneous,orozco1997reduced,choi2012simultaneous,choi2015practical,pearson2012new,rees2010all,barker2016fast}
computationally intractable, as the dimension of the system of equations
arising in such applications scales with the temporal dimension of the
problem. 
%In addition, computing adjoint-based sensitivities required for error
%estimation or nested analysis and design (NAND) in PDE-constrained
%optimization incurs storage and computational costs that scale with the
%temporal dimension of the problem.  
Further, rigorous error bounds for these ROMs
typically grow exponentially in time
\cite{rathinam:newlook,knezevic2011reduced,nguyen2009reduced,carlbergJCP},
which renders certification challenging. This work aims to devise a
model-reduction methodology that enables significant reduction in both the
spatial and temporal dimensions of the dynamical system, while simultaneously
producing error bounds that exhibit slower growth in time. 

Several attempts have been made to address this temporal-complexity bottleneck
in model reduction. First, several authors have
	demonstrated that larger stable time steps (and thus a smaller number of
	time instances) can be taken with a ROM relative to the high-fidelity model
	in the case of explicit time integration
	\cite{krysl_nl_rom_dynam_struct_01, lucia2004reduced, taylor2010real}. 
	However, this approach is not always feasible, as many problems
	(e.g., compressible fluid dynamics, chemical kinetics) exhibit stiff
	dynamics that require implicit time integration, where the time step is
	limited by accuracy rather than stability. Increasing the time step in these
	contexts could significantly degrade time-discretization accuracy.

Time-parallel methods (e.g., parareal \cite{lions2001parareal}, PITA
\cite{farhat2003time}, and MGRIT \cite{mgrit}) aim to reduce the (serial) wall time
incurred by a fine temporal discretization. These approaches enable
dynamical-system simulations to be parallelized in the temporal domain, and
are well suited for reduced-order models, as spatial parallelism alone quickly
saturates for such low-dimensional models. However, while time-parallel
methods can reduce the wall-time of such simulations, they do not reduce
temporal dimensionality; in fact, they increase the total computational cost
of simulations (as measured in core--hours).

More recently, a `forecasting' approach was proposed that employs time-domain
data (i.e., snapshot-matrix right singular vectors) generated during the
offline stage of model reduction to produce accurate forecasts of the solution
during online ROM simulations via gappy proper orthogonal decomposition (POD)
\cite{sirovichOrigGappy}. These forecasts can be used (1) to generate accurate
initial guesses for the Newton solver at each time step
\cite{carlberg2015decreasing}, or (2) as an accurate coarse propagator to
accelerate convergence of time-parallel methods \cite{carlberg2016data}. While
both approaches reduce the computational cost incurred by time integration
(by reducing the total number of Newton iterations and the wall time,
respectively), neither directly reduces the temporal dimension of the
ROM.

Alternatively, space--time ROMs have been devised in the reduced basis
\cite{urban2012new,urban2014improved,yanoPateraUrban,yanoReview},
POD--Galerkin \cite{volkwein2006algorithm,baumann2016space}, and ODE-residual
minimization \cite{constantineResMin} contexts. These approaches successfully
reduce the temporal dimension of the underlying model by performing projection
with a low-dimensional space--time basis. In addition,
these methods \reviewerBRtwo{can
remove} spurious temporal modes (e.g., unstable growth, artificial
dissipation) from the state space, which can in principle lead to more
accurate long-time responses.
Further, space--time reduced-basis
ROMs \cite{urban2012new,urban2014improved,yanoPateraUrban,yanoReview} are
equipped with error bounds that are observed to grow linearly (rather than
exponentially) in the final time. While these approaches are quite promising,
they exhibit several drawbacks in terms of applicability to general
large-scale nonlinear dynamical
systems.  First, the space--time reduced-basis approaches require a
space--time finite-element discretization for the high-fidelity model. Such
discretizations are uncommon, as most computational models used in practice
are constructed via spatial discretization (e.g., with a finite
difference, finite volume, or finite element method) followed by time integration (e.g., with a linear multistep or
Runge--Kutta scheme). Second, these space--time ROM
approaches (with the exception of a collocation-like approach proposed in
Ref.~\cite{constantineResMin}) provide no mechanism for complexity reduction
(i.e., hyper-reduction), which precludes these techniques from reducing the
computational complexity in the presence of general nonlinearities.  Further,
the above approaches (with the exception of Ref.~\cite{baumann2016space})
compute only a single space--time basis vector per training simulation. 
This can severely limit the dimensionality
(and accuracy) of the resulting space--time ROM
in the case of
large-scale nonlinear dynamical-system models, where the number of training simulations may be limited by
computational-cost considerations.

To this end, we propose a novel space--time least-squares Petrov--Galerkin
(\methodAcronym) method that combines advantages of the above space--time ROM
approaches, as it: (1) reduces the spatial and temporal dimensions of the
dynamical system; 
%(2) \reviewerBRtwo{can remove} spurious temporal modes from the ROM response;
(2) is equipped with \reviewerBRthree{\textit{a priori}} error bounds that 
\reviewerBRthree{bound the solution error by the best space--time
approximation error and whose stability constants}
exhibit subquadratic growth in time;
(3) is applicable to general nonlinear dynamical-system models;  (4) is
equipped with hyper-reduction to reduce the complexity in the presence of
general nonlinearities; and (5) can extract multiple space--time basis vectors
from each training simulation via tensor decomposition.  To realize these
advantages, the approach adopts aspects of both the forecasting and
space--time ROM approaches described above.

%we propose to apply a discrete-optimal space--time projection to
%nonlinear ODE models that have been discretized in time using (implicit or
%explicit) linear multistep schemes.  The method---which ---adopts a
%least-squares Petrov--Galerkin (LSPG) ROM formulation
%\cite{CarlbergGappy,carlbergGalDiscOpt} to perform discrete-optimal
%spatiotemporal projection.

The original spatial-projection-based LSPG method
\cite{CarlbergGappy,carlbergJCP,carlbergGalDiscOpt} performed the following
steps: (1) apply temporal discretization\footnote{When explicit time
integration is used, LSPG projection is equivalent to Galerkin projection
\cite{carlbergGalDiscOpt}.} to the system of ODEs characterizing the
high-fidelity model, (2) introduce a low-dimensional \textit{spatial} trial
subspace, and  (3) compute the solution in the spatial subspace that minimizes
(in a weighted $\ell^2$-norm) the discrete residual arising at each time step.
This approach does not reduce temporal dimensionality, as the ROM and
high-fidelity model time steps are typically the same.  Instead, the proposed
\methodAcronym\ method executes the following steps: (1) apply time
integration to the system of nonlinear ODEs with an implicit or explicit
linear multistep method, (2) introduce a low-dimensional \textit{space--time}
trial subspace, and (3) compute the solution in the space--time trial subspace
that minimizes the discrete residual over all space and time in a weighted
$\ell^2$-norm.	This norm can be selected to enable hyper-reduction based on
both collocation and gappy POD; this is in analogy to collocation
\cite{LeGresleyThesis,astrid2007mpe,ryckelynck2005phm} and gappy POD
\cite{astrid2007mpe,carlbergJCP} methods applied to spatial-projection-based
ROMs.  

Specific contributions of this work include:
\begin{itemize} 
\item A novel \methodAcronym\ model-reduction  method for
parameterized nonlinear dynamical systems
(Section \ref{sec:spaceTimeLSPG}), including choices of weighting matrices to
enable hyper-reduction (Section \ref{sec:spacetimehyper}).
\item Several strategies for computing the `ingredients' characterizing the
\methodAcronym\ method: the space--time trial subspace via tensor
decomposition (Section \ref{sec:spacetimetrialconstruct}), the
space--time residual basis in the case of space--time GNAT (Section \ref{sec:spacetimeresidual}),
the sampling matrix to enable space--time hyper-reduction (Section \ref{sec:constructSamplingMatrix}),
and the initial guess for the Gauss--Newton solver used to compute the
\methodAcronym\ solution (Section \ref{sec:initialGuess}).
\item \textit{A priori} error bounds that enable the error in the
\methodAcronym\ solution to be bounded by the best approximation error as
measured in the $\ell^2$-norm over all time (Theorem \ref{thm:aprioritwo}),
and the $\ell^\infty$-norm over all time (Theorem \ref{thm:apriori}).
Critically, the stability constants for these bounds grow linearly and
subquadratically in time, respectively (Remark \ref{rem:stabilityConstant}).
\item \textit{A posteriori} error bounds that enable the error in the
\methodAcronym\ solution to be bounded by the value of the objective function
minimized by the method (Corollary \ref{eq:aposterioribound}).
\item Numerical experiments that demonstrate the ability of the method to
produce significant computational-cost savings relative to existing
spatial-projection-based nonlinear ROMs without sacrificing accuracy (Section
\ref{sec:experiments}).
\end{itemize}

Ref.\ \cite{constantineResMin}, which also proposed a space--time
residual-minimizing projection applicable to parameterized nonlinear dynamical
systems, is perhaps the most closely related work to the proposed technique.  Our work
can be distinguished from that contribution in several ways. First, our
approach applies residual minimization to the discretized ODE (i.e.,
O$\Delta$E) rather than the time-continuous ODE over all time. This
facilitates deriving error bounds with respect to the (fully discrete)
full-order-model solution (Section \ref{sec:error}); 
Ref.~\cite{constantineResMin} did not provide
error bounds for the
residual-minimizing approximation (it provides an error bound only for the best
linear-subspace approximation). Also, Ref.~\cite{constantineResMin} enforces
the sum of generalized-coordinate values to equal one; our method does not
require such a constraint, which can enable lower objective-function values.
Further, our approach enables multiple space--time basis vectors to be
extracted from each training simulation via tensor decomposition (Section
\ref{sec:spacetimehyper}); Ref.~\cite{constantineResMin} computes only a
single space--time basis vector from each training simulation, which can
severely limit the dimensionality of the space--time basis in practice.
Further, our proposed approach provides several mechanisms for enabling
hyper-reduction, i.e., complexity reduction of the low-dimensional model
(Section \ref{sec:spacetimehyper}); Ref.~\cite{constantineResMin} proposes one
approach, which is analogous to the space--time collocation method described
in Section \ref{sec:spacetimecoll}.

%In the remainder of this paper, we denote matrices by capitalized bold
%letters, vectors by lowercase bold letters, and scalars by unbolded letters.
%The columns of a matrix $\boldsymbol{A}\in \RR{m\times {\timestepit}}$ are
%denoted by $\boldsymbol{a}_{i}\in \RR{m}$, $i\innat{{\timestepit}}$ with
%$\nat{a}\defeq \{1,\ldots , a\}$ such that $\boldsymbol{A}\equiv \left[
%\boldsymbol{a}_{1}\ \cdots \ \boldsymbol{a}_{\timestepit }\right] $. {We also
%define $\natZero{a}\defeq\{0,\ldots, a\}$.} The scalar-valued matrix elements
%are denoted by $a_{ij}\in \RR{}$ such that $\boldsymbol{a}_{j}\equiv \left[
%a_{1j}\ \cdots \ a _{mj}\right] ^{T}$, $j\innat{{\timestepit}}$. A superscript
%denotes the value of a variable at that time instance, e.g.,
%$\stateArg{\timestepit }$ is the value of $\state $ at time $\timestepit \dt
%$, where $\dt $ is the time step.

We proceed by describing the time-continuous and time-discrete representations
of the full-order model in Section \ref{sec:FOM}, followed by a summary of the
previously developed spatial-projection-based LSPG method in Section \ref{sec:spatialLSPG}. Then, we
present the \methodAcronym\ method in Section \ref{sec:spaceTimeLSPG},
followed by proposals for the ingredients characterizing the method in Section
\ref{sec:ingredients}. Section \ref{sec:error} provides error analysis, 
Section \ref{sec:experiments} reports numerical experiments, and Section
\ref{sec:conclusions} concludes the paper.

\section{Full-order model}\label{sec:FOM}
We begin by deriving time-continuous (ODE) and time-discrete (O$\Delta$E)
formulations of the full-order model (FOM).
\subsection{Time-continuous representation} 
 We consider the FOM to be a parameterized nonlinear dynamical system
 characterized by a parameterized system of nonlinear ODEs
 \begin{equation} \label{eq:fom}
  \frac{d\solTimeContinuous}{dt} = \flux(\solTimeContinuous,t; \param),\quad\quad
  \solTimeContinuous(0;\param) = \solinit(\param),
 \end{equation} 
 where $t\in[0,\totaltime]$ denotes time with
$\totaltime\in\RRplus{}$ denoting the final time, and
 $\solTimeContinuous(t;\param)$ denotes the time-dependent, parameterized state
 implicitly defined as the solution to problem \eqref{eq:fom}
with
 $\solTimeContinuous:\timeDomain\times \paramDomain\rightarrow \RR{\nspacedof}$
 and $\solTimeContinuous(\cdot;\param)\in\RR{\nspacedof}\otimes \timeSpace$.
 Here,
 $\timeSpace$ denotes the set of sufficiently smooth functions 
 from $\timeDomain$ to $\RR{}$ 
 (e.g., $\timeSpace = H^1(\timeDomain)$)
 under consideration \reviewerA{and $\otimes$ denotes the tensor product}.
 Further, 
 $\flux: \RR{\nspacedof} \times [0,\totaltime] \times \paramDomain \rightarrow
 \RR{\nspacedof}$ with $(\solDummy,\timeDummy;\paramDummy)\mapsto\flux(\solDummy,\timeDummy;\paramDummy)
 $ denotes the velocity, which we assume to be nonlinear in at least its
 first argument,
 $\solinit:\paramDomain\rightarrow \RR{\nspacedof}$ denotes the initial state,
 and $\param \in \paramDomain$ denotes the parameters with parameter domain
 $\paramDomain\subseteq\RR{\nparam}$.

\subsection{Time-discrete representation: linear multistep methods}
We now introduce a \reviewerA{(generally nonuniform)} time discretization characterized by time step
$\dtArg{n}\in\RRplus{}$ and time instances $\timeArg{n} = \timeArg{n-1} +
\dtArg{n}$,
$n\innat{\ntimedof}$ with $\timeArg{0} = 0$, $\ntimedof\in\natNo$, and
$\nat{N}\defeq\{1,\ldots,N\}$.\footnote{\reviewerA{Note that by assuming a
fixed time discretization, we do not allow for adaptive time-step selection;
enabling the proposed \methodAcronym\ method to be applied in the context of 
adaptive time stepping is the subject of future research.}}
%We denote the set of time instances associated with this discretization as
%$\timestepset = \{\timeArg{n}\}_{n=0}^{\ntimedof}\subset\timeDomain$.
Applying a linear $k$-step method\footnote{\reviewerA{We consider only linear
multistep methods for simplicity of presentation. One could develop the
proposed \methodAcronym\ method for alternative time integrators in a rather
straightforward manner;
see, e.g., 
Ref.~\cite{carlbergGalDiscOpt}, which develops LSPG models for explicit, fully implicit, and
diagonally implicit Runge--Kutta schemes.}} to numerically solve Eq. \eqref{eq:fom}
using this discretization yields an
O$\Delta$E, which is characterized by the following system of 
 nonlinear algebraic equations to be solved for the numerical solution
 $\solFuncArg{n}\in\RR{\nspacedof}$ at each time instance:
%\begin{align} 
%\resk(\solFuncArg{n},\ldots,\solFuncArg{n-k};\timeArg{n},\ldots,\timeArg{n-k};\param)
%= \zerobold.
%\end{align} 
 \begin{align} \label{eq:LMMresidual}
   \resn(\solFuncArg{n},\ldots,\solFuncArg{n-\kn };\param)
	 = \zerobold,\quad n=1,\ldots,\ntimedof.
 \end{align} 
Here, $\kn(\leq n)$ denotes the number
 of steps used by the linear multistep method at time instance $n$ and
 the residual is defined as
%\begin{align}\label{eq:residual} 
%\begin{split}
%\resk&:(\solDummy^0,\ldots,\solDummy^k;\timeDummy^0,\ldots,\timeDummy^k;\paramDummy) \mapsto
%\sum_{j=0}^k\alpha_j\solDummy^{j}
%-\dt
%\sum_{j=0}^k\beta_j\flux(\solDummy^j,\timeDummy^j;\paramDummy)\\
%&:\RR{\nspacedof}\times\cdots\times\RR{\nspacedof}\times\timeDomain\times\cdots\timeDomain\times\paramDomain\rightarrow\RR{\nspacedof},
%\end{split}
%\end{align} 
 \begin{align}\label{eq:residual} 
 \begin{split}
  \resn&:(\solDummy^0,\ldots,\solDummy^{\kn};\paramDummy) \mapsto
	\sum_{j=0}^{\kn }\alpha_j^n\solDummy^{j}
	-\dtArg{n}
							\sum_{j=0}^{\kn }\beta_j^n\flux(\solDummy^j,t^{n-j};\paramDummy) \\ 
&:\RR{\nspacedof}\otimes\RR{\kn+1}\times\paramDomain\rightarrow\RR{\nspacedof},
 \end{split}
 \end{align} 
%\begin{align} 
%\sum_{j=0}^k \alpha_j\soli{n-j} = \dt\sum_{j=0}^k\beta_j \flux(\soli{n-j},t^{n-j};\param),
%\end{align} 
where coefficients $\alpha_j^n, \beta_j^n\in\RR{}$, $j=0,\ldots,\kn$ define
 a particular linear multistep scheme, 
 $\alpha_0^n \neq 0$, and
 $\sum_{j=0}^{\kn}\alpha_j^n = 0$ is necessary for consistency. We note that the
 numerical solution satisfies
 \begin{equation}\label{eq:FOMcontinuous}
 \solFuncOnlyParam\in\spatialSpace\otimes\timeDiscreteFuncSpace
 \end{equation}
 with $\timeDiscreteFuncSpace$
 the set of functions from $\timestepset$ to $\RR{}$. We refer to
 $\spatialSpace\otimes\timeDiscreteFuncSpace$ 
 as the `FOM trial space' in which
 solutions are sought. An isomorphism exists between
 $\timeDiscreteFuncSpace$ and $\timeDiscreteSpace$ provided by the
 (invertible) function
$\unrollfuncNo$ that `unrolls' time according to the
 time discretization as
 \begin{align}\label{eq:unroll}
 \begin{split}
 \unrollfuncNo&:{\stateContinuousDummyEntryNo} \mapsto \left[
 \stateContinuousDummyEntryNo(\timeArg{1})\
 \cdots\
 \stateContinuousDummyEntryNo(\timeArg{\nTimestepsFine})
 \right]\\
 &:\RR{p}\otimes\timeDiscreteFuncSpace\rightarrow\RR{p}\otimes\timeDiscreteSpace,
 \end{split}
 \end{align}
 where $p\in\natNo$ is an arbitrary dimension.
This yields an equivalent (discrete) representation of the FOM space as
 \begin{equation} \label{eq:FOMdiscrete}
%\solFullDiscreteOnlyParam \defeq 
\unrollfunc{\solFuncOnlyParam}=\left[\solArg{1}\ \cdots\
\solArg{\nTimestepsFine}\right]\in
\RR{\nspacedof}\otimes\timeDiscreteSpace.
  \end{equation}
\section{Least-squares Petrov--Galerkin method}\label{sec:spatialLSPG} 
This section describes the trial subspace (Section \ref{sec:LSPGsubspace}) and projection
(Section \ref{sec:LSPGprojection}) employed by the original
spatial-projection-based LSPG method \cite{CarlbergGappy,carlbergJCP,carlbergGalDiscOpt}.
\subsection{Spatial trial subspace}\label{sec:LSPGsubspace}
The original LSPG method
 applies spatial
projection using a subspace $\spatialSubspace \defeq
\Span{\basisvecspace_i}_{i=1}^\nbasisspace
\subseteq \RR{\nspacedof}$ with
$\dim(\spatialSubspace)=\nbasisspace\leq\nspacedof$ (hopefully with
$\nbasisspace\ll\nspacedof$).  Using this subspace, the LSPG
method approximates the numerical solution at each time instance
$\timeArg{n}$, $n\innat{\ntimedof}$ as $ \solFuncArg{n}\approx\solapproxFuncArg{n}
\in\solref+\spatialSubspace $ or equivalently
\begin{equation}\label{eq:spatialLSPGsolution}
\solapproxFuncArg{n}=
\solref + \sum_{i=1}^\nbasisspace\basisvecspace_i\redsolapproxEntryFuncArg{i}{n}
\end{equation}
where
$\redsolapproxEntryFuncOnlyParam{i}\in\RR{}\otimes\timeDiscreteFuncSpace$ with
$\redsolapproxEntryFuncArgGen{i}{0}=0$,
$i\innat\nbasisspace$
denotes the generalized coordinates.
%and
%$\solref\in\spatialSpace$ denotes a reference solution that can be defined to ensure
%consistency \cite{CarlbergGappy,carlbergJCP}.\footnote{While the
 %reference state $\solref$ may depend on the parameters $\param$, we omit this
 %dependence for notational simplicity.}
This approach is equivalent to enforcing
 the approximated numerical solution $\solapprox(t;\param)$ with 
 $\solapprox:\timeDomain\times \paramDomain\rightarrow \RR{\nspacedof}$
 to reside in \reviewerA{an affine} `spatial trial subspace'
\begin{equation} \label{eq:LSPGcontinuous}
\solapprox(\cdot;\param)\in\solref\otimes
\onesFunc+\spatialSubspace\otimes
\timeDiscreteFuncSpace\subseteq\spatialSpace\otimes\timeDiscreteFuncSpace,
\end{equation} 
where $\onesFunc\in\timeDiscreteFuncSpace$ is defined as
$\onesFunc:\timestepset\rightarrow
1$.
Noting that $\ones_{\ntimedof}=\unrollfunc{\onesFunc}$, where $\ones_{p}$
denotes a $p$-vector of ones, we also have 
\begin{equation} \label{eq:LSPGdiscrete}
%\solapproxFullDiscreteOnlyParam 
%\defeq 
\unrollfunc{\solapproxFuncOnlyParam}
%=\left[\solapproxArg{1}\ \cdots\
%\solapproxArg{\nTimestepsFine}\right]
\in
\solref \otimes \ones_{\ntimedof} + \spatialSubspace \otimes
	 \timeDiscreteSpace \subseteq \RR{\nspacedof} \otimes \timeDiscreteSpace.
\end{equation} 

\begin{remark}[LSPG projection does not reduce the temporal dimension]
	\label{rem:lspgTemporalDim}
%Eqs.~\eqref{eq:FOMcontinuous} and \eqref{eq:LSPGcontinuous}, 
Comparing \eqref{eq:FOMdiscrete} and \eqref{eq:LSPGdiscrete} reveals that
LSPG projection reduces spatiotemporal dimension of the problem from
$\dim(\RR{\nspacedof} \otimes \timeDiscreteSpace)=\ndof\ntimedof$ to
$\dim(\solref \otimes \ones_{\ntimedof} + \spatialSubspace \otimes
\timeDiscreteSpace)=\nbasisspace\ntimedof$. Thus, while the spatial dimension
has been reduced from $\ndof$ to $\nbasisspace$, the temporal dimension
$\ntimedof$ has not been reduced. This can preclude significant
computational-cost savings when the original problem is characterized by a
long time interval $\totaltime$, or if \reviewerA{small time steps $\dtArg{n}$
	are} needed for
accuracy (e.g., for stiff dynamical systems).
\end{remark}
 
\subsection{Spatial projection}\label{sec:LSPGprojection}
The LSPG ROM computes an approximate solution by sequentially minimizing the discrete residual arising at each
	 time instance, i.e.,
\begin{equation} \label{eq:lspggnat}
  \solapproxArg{n} = \underset{\solDummyOpt\in\solref+\spatialSubspace}{\arg\min}\left \|\weightmat 
\resn(\solDummyOpt,\solapproxArg{n-1},\ldots,\solapproxArg{n-\kn };\param)	
	\right \|_2^2,\quad n\innatSeq{\nTimestepsFine},
\end{equation}
or equivalently
\begin{equation} \label{eq:lspggnatReduced}
  \redsolapproxArg{n} = \underset{\solRedDummyOpt\in\RR{\nbasisspace}}{\arg\min}\left \|\weightmat 
\resn(\solref + \basismatspace\solRedDummyOpt,\solapproxArg{n-1},\ldots,\solapproxArg{n-\kn };\param)	
	\right \|_2^2,\quad n\innatSeq{\nTimestepsFine},
\end{equation}
where
$\redsolapproxFuncOnlyParam\equiv[\redsolapproxEntryFuncOnlyParam{1}\ \cdots\
\redsolapproxEntryFuncOnlyParam{\nbasisspace}]^T\in\RR{\nbasisspace}\otimes\timeDiscreteFuncSpace$.
Here,
 $\weightmat\in \RR{\weightmatRows\times\ndof}$ is a weighting matrix, where
 $\nbasisspace\leq\weightmatRows(\leq\ndof)$ is necessary for the residual
 Jacobian in the nonlinear least-squares problem 
 \eqref{eq:lspggnat}--\eqref{eq:lspggnatReduced}
 to be 
 nonsingular.
 Examples of weighting matrices include 
 $\weightmat = \identity{\ndof}$ in the case of unweighted LSPG, where
 $\identity{p}$ denotes the $p\times p$ identity matrix; however, this choice precludes
 computational-cost savings, as all $\ndof$ elements of the residual vector
 $\resn$ (as well as its Jacobian) must be computed during
 each iteration when solving nonlinear least-squares problem
 \eqref{eq:lspggnat}--\eqref{eq:lspggnatReduced}. To reduce the
 computational complexity,
 a `hyper-reduction' approach is required that necessitates computing only a subset of residual
 elements. Particular weighting-matrix choices that lead to
 hyper-reduction include $\weightmat =
 \samplemat\in\{0,1\}^{\nressample\times\nspacedof}$, where
 $\nressample\leq\nspacedof$ and $\samplemat$
 comprises selected rows of the identity matrix $\identity{\ndof}$ in the case of collocation
 \cite{LeGresleyThesis}; and 
$\weightmat = (\samplemat \basismatres)^{+}\samplemat\in\RR{\nbasisres\times\nspacedof}$,
 where $\basismatres\in\orthomat{\nspacedof}{\nbasisres}$ 
 denotes a 
 basis for the residual and a superscript $+$ denotes the Moore--Penrose
 pseudoinverse in the case of GNAT \cite{carlbergJCP}.
 Here, 
 $\nbasisspace\leq\nbasisres\leq\nressample(\leq\ndof)$
 is necessary for the residual Jacobian to be nonsingular, 
and
 $\orthomat{n}{k}$ denotes
 the Stiefel manifold: 
 the set of orthogonal $k$-frames in $\RR{n}$. 

\section{Space--time least-squares Petrov--Galerkin method}\label{sec:spaceTimeLSPG}
We now derive the proposed space--time least-squares Petrov--Galerkin
(\methodAcronym)
projection method. 
We begin by specifying the space--time trial subspace in Section
\ref{sec:STLSPGsubspace}, 
followed by a description of the space--time least-squares Petrov--Galerkin
projection process in Section
\ref{sec:STLSPGprojection}. Then, Section \ref{sec:spacetimehyper} describes
choices for the weighting matrix to enable hyper-reduction.

\subsection{Space--time trial subspace}\label{sec:STLSPGsubspace}
To reduce both the spatial and temporal dimensions of the FOM, in
analogy to
Eq.~\eqref{eq:LSPGcontinuous} 
we enforce the approximated numerical solution 
$\solapproxSTFunc$
to reside in \reviewerA{an affine} `space--time trial subspace'
\begin{equation}\label{eq:spacetimeLSPGcontinuous}
\solapproxSTFuncOnlyParam\in
\spacetimesubspaceFunc
\subseteq\spatialSpace\otimes\timeDiscreteFuncSpace
%\subseteq\underbrace{\solref\otimes
%\onesFunc+\spatialSubspace\otimes
%\timeDiscreteFuncSpace}_\text{spatial 
%trial
%subspace}
\end{equation}
where 
$\spacetimesubspaceFunc\defeq\solref\otimes
\onesFunc+\Span{\stbasisvecFunc{i}}_{i=1}^\nbasisst\subseteq\spatialSpace\otimes\timeDiscreteFuncSpace$
with
$\dim(\spacetimesubspaceFunc)=\nbasisst\ll\nspacedof\ntimedof$, and we
enforce
$\stbasisvecFunc{i}(0) =\zerobold$, $i\innat{\nbasisst}$ such that
$\solapproxSTFunc(0;\param)=\solref$.

Thus, at a given time instance, the space--time LSPG method approximates the numerical
solution at each time instance
$n\innat{\ntimedof}$ as 
$
\solFuncArg{n}\approx\solapproxSTFuncArg{n}\in\spacetimesubspaceFunc
$
or equivalently
\begin{align}\label{eq:spacetimeExpansion}
\begin{split}
\solapproxSTFuncArg{n}&= \solref +
\sum_{i=1}^{\nbasisst}
\stbasisvecFunc{i}(t^n)\redsolapproxSTArg{i} 
\end{split}
\end{align}
%\begin{align}\label{eq:spacetimeExpansion}
%\begin{split}
%\solapproxSTFuncArg{n}&= \solref +
%\sum_{i=1}^{\nbasisspace}\sum_{j=1}^{\nbasistime^i}
%\basisvecspacei{i}
%\basisvectimeFuncij{i}{j}(t^n)\redsolapproxSTArg{i}{j} 
%%=\solref +
%%\sum_{i=1}^{\nbasisspace}
%%(\basisvectimeFuncVeci{i}(t^n)\otimes
%%\basisvecspacei{i})
%%\redsolapproxSTVecArg{i},
%\end{split}
%\end{align}
where 
$\redsolapproxSTArg{i}\in\RR{}$, $i\innat{\nbasisst}$
denotes the generalized
coordinate of the \methodAcronym\ solution. Comparing
Eqs.~\eqref{eq:spatialLSPGsolution} and \eqref{eq:spacetimeExpansion} reveals
that the space--time trial subspace enables time dependence of the
approximated solution to be moved from the generalized coordinates to the
basis vectors; this enables fewer generalized coordinates to be computed in
order to characterize the complete space--time solution.

Introducing 
$\stbasisvec{i}\defeq
\unrollfunc{\stbasisvecFunc{i}}\in\RR{\ndof}\otimes\timeDiscreteSpace$
and
$\spacetimesubspace\defeq\solref\otimes
\ones_{\ntimedof}+\Span{\stbasisvec{i}}_{i=1}^{\nbasisst}$,
%$\basisvectimeij{i}{j}\defeq
%\unrollfunc{\basisvectimeFuncij{i}{j}}\in\timeDiscreteSpace$ and 
%$\temporalSubspaceArg{i} :=
%\Span{\basisvectimeij{i}{j}}_{j=1}^{\nbasistime^i} \subseteq \timeDiscreteSpace$,
%and $\nbasistime^i\leq \ntimedof$	for $i\innat{\nbasisspace}$, 
we can also
write the space--time trial subspace in discrete form as
%define 
\begin{equation} \label{eq:spacetimeLSPGdiscrete}
%\solapproxSTFullDiscreteOnlyParam 
%\defeq 
\unrollfunc{\solapproxSTFuncOnlyParam}
%=\left[\solapproxSTArg{1}\ \cdots\
%\solapproxSTArg{\nTimestepsFine}\right]
\in
\spacetimesubspace
%\subseteq 
%\solref \otimes \ones_{\ntimedof} + \spatialSubspace \otimes
	 %\timeDiscreteSpace 
	 \subseteq \RR{\nspacedof} \otimes \timeDiscreteSpace.
\end{equation} 
\begin{remark}[Space--time LSPG projection reduces the temporal dimension]\label{rem:stlspgreducedim}
Comparing \eqref{eq:FOMdiscrete} and \eqref{eq:spacetimeLSPGdiscrete}
reveals that the proposed space--time LSPG trial subspace reduces the spatiotemporal dimension of the
problem from $\dim(\RR{\nspacedof} \otimes \timeDiscreteSpace)=\ndof\ntimedof$ to
$\dim(\spacetimesubspace) = \nbasisst$. Because the spatiotemporal dimension
$\nbasisst$ can be chosen to be independent
of both the spatial and temporal dimensions $\nspacedof$ and $\ntimedof$,
respectively, the proposed method can reduce both the spatial and temporal
dimensions of the full-order model. 
\end{remark}
\begin{remark}[Space--time trial subspace \reviewerBRtwo{can remove} spurious temporal
modes]\label{rem:spurious}
Restricting the \methodAcronym\ state to lie in the space--time trial subspace
$\solref\otimes \onesFunc+ \spacetimesubspaceFunc$
\reviewerB{\reviewerBRtwo{can enable} spurious temporal modes (e.g., spurious
time growth or dissipation) to be removed} from the set of possible solutions.
\reviewerB{\reviewerBRtwo{More precisely}, if the subspace
	$\spacetimesubspaceFunc$ is computed from training data as will be described
	in Section \ref{sec:spacetimetrialconstruct}, then this subspace will
	contain only temporal modes that have been observed during the training
simulations.
%; the \methodAcronym\ solution will thus be prevented from
%exhibiting previously unobserved temporal behavior.
}
\end{remark}
\subsection{Space--time least-squares Petrov--Galerkin
projection}\label{sec:STLSPGprojection}
To derive the \methodAcronym\ projection, we begin by defining
\begin{align}
\begin{split}
\res&:(\timeDummy^n;\solDummy;\paramDummy)\mapsto
\resArg{n}(\solDummyFuncArgDummy{n},\ldots,\solDummyFuncArgDummy{n-\knDummy };\paramDummy)\\
&:\timestepset\times\RR{\ndof}\otimes \timeDiscreteFuncSpace\times
\paramDomain\rightarrow \RR{\ndof}
\end{split}\\
\begin{split}
\resRed&:(\timeDummy^n;\solRedDummy;\paramDummy)\mapsto
\res(\timeDummy^n;
\solref +
\sum_{i=1}^{\nbasisst}
\stbasisvecFunc{i}(\cdot)
\solRedDummyScalarArg{i};\paramDummy)\\
&:\timestepset\times\RR{\nbasisst}\times
\paramDomain\rightarrow \RR{\ndof}.
\end{split}
\end{align}
Note that 
\reviewerBRtwo{the space--time residuals}
$\res(\cdot;\solDummy;\paramDummy),\resRed(\cdot;\solRedDummy;\paramDummy)
\in\RR{\ndof}\otimes\timeDiscreteFuncSpace $ \reviewerBRtwo{are defined
from the O$\Delta$E residual $\resArg{n}$; as such, they are defined
directly from the integrator used to perform time discretization for the FOM}.
We now introduce a vectorization function 
 \begin{align}\label{eq:vectorize}
 \begin{split}
 \vectorizeNone&:{\stateContinuousDummyEntryNo} \mapsto
 \text{vec}(\unrollfunc{\stateContinuousDummyEntryNo})\\
 &:\RR{p}\otimes\timeDiscreteFuncSpace\rightarrow\RR{p\ntimedof},
 \end{split}
 \end{align}
 and define the vectorized residual as a function of the full state
\begin{align}
\begin{split}
\resst&:(\solDummy;\param) \mapsto
\vectorize{\res(\cdot;\solDummy;\param)}\\
&:\RR{\ndof}\otimes\timeDiscreteFuncSpace
\times
\paramDomain\rightarrow\RR{\ndof\ntimedof}
\end{split}
\end{align}
and as a function of the generalized coordinates
\begin{align}
\begin{split}
\resRedst&:(\solRedDummy;\paramDummy) \mapsto
\vectorize{\resRed(\cdot;\solRedDummy;\paramDummy)}\\
&:\RR{\nbasisst}\times\paramDomain\rightarrow\RR{\ndof\ntimedof}.
\end{split}
\end{align}
 We define the inner product
$(\stateContinuousDummyEntryNo,\stateTwoContinuousDummyEntryNo)_\metric\defeq
\vectorize{\stateTwoContinuousDummyEntryNo}^T\metric\vectorize{\stateContinuousDummyEntryNo
}
$ and associated norm $\|\stateContinuousDummyEntryNo\|_\metric =
\sqrt{(\stateContinuousDummyEntryNo,\stateContinuousDummyEntryNo)_\metric}$
for
	$\stateContinuousDummyEntryNo,\stateTwoContinuousDummyEntryNo\in\RR{p}\otimes\timeDiscreteFuncSpace$
	and $\metric\in\SPSD{p\ntimedof}$, where $\SPSD{p}$ denotes the set of $p\times p$
	symmetric positive semidefinite matrices; we also define
the inner product
$(\stateContinuousDummyEntryNo,\stateTwoContinuousDummyEntryNo)_2\defeq\vectorize{\stateTwoContinuousDummyEntryNo}^T\vectorize{\stateContinuousDummyEntryNo
} $ and the associated norm
$\|\stateContinuousDummyEntryNo\|_2 =
\sqrt{(\stateContinuousDummyEntryNo,\stateContinuousDummyEntryNo)_2}$.\footnote{The
metric will be rank deficient if we employ hyper-reduction as described in
Section \ref{sec:spacetimehyper}.}
Now, we propose computing the \methodAcronym\ solution by
 minimizing the residual in a weighted $\ell^2$-norm as
\begin{align} \label{eq:t-lspgReducedLargerFull}
\begin{split}
 \solapproxSTFuncOnlyParam =
	\underset{
\solDummyOpt\in\spacetimesubspace
		}{\arg\min}
    \left \|
		\res(\cdot;
\solDummyOpt
;\param)
		\right \|_{\weightmatst^T\weightmatst}^2
 =
	\underset{
\solDummyOpt\in\spacetimesubspace
		}{\arg\min}
    \left \|\weightmatst
		\resst(
\solDummyOpt;\param)
		\right \|_2^2
		,
\end{split}
\end{align}
where $\weightmatst\in\RR{\weightmatstRows\times\ndof\ntimedof}$ 
is a space--time weighting matrix and
$\nbasisst\leq\weightmatstRows(\leq\ndof\ntimedof)$ is necessary for the
residual Jacobian in the nonlinear least-squares problem
\eqref{eq:t-lspgReducedLargerFull} to be nonsingular.
We can also write
Problem \eqref{eq:t-lspgReducedLargerFull} in terms of the generalized
coordinates as
%\begin{equation} \label{eq:t-lspgReducedLarger}
% \redsolapproxST(\param) =
%	\underset{(\solDummyScalar_{ij})}{\arg\min}
%    \left \|\weightmatst
%		\resst(
%\solref +
%\sum_{i=1}^{\nbasisspace}\sum_{j=1}^{\nbasistime^i}
%\basisvecspacei{i}
%\basisvectimeFuncij{i}{j}(t^0)\solDummyScalar_{ij},
%		\ldots,
%\solref +
%\sum_{i=1}^{\nbasisspace}\sum_{j=1}^{\nbasistime^i}
%\basisvecspacei{i}
%\basisvectimeFuncij{i}{j}(t^\ntimedof)\solDummyScalar_{ij};\param)
%		\right \|_2^2,
%\end{equation}
%
%Further, defining enables Problem \eqref{eq:t-lspgReducedLarger} to be simplified to 
\begin{align} \label{eq:t-lspgReducedLargerSimplify}
\begin{split}
 \redsolapproxST(\param) &=
	\underset{\solRedDummyOpt\in\RR{\nbasisst}}{\arg\min}
    \left \|\resRed(\cdot;\solRedDummyOpt;\param)
		\right \|_{\weightmatst^T\weightmatst}^2
	=\underset{\solRedDummyOpt\in\RR{\nbasisst}}{\arg\min}
    \left \|\weightmatst
\resRedst(\solRedDummyOpt;\param)
		\right \|_2^2,
\end{split}
\end{align}
where 
$\redsolapproxST\equiv[ \redsolapproxSTArg{1}\ \cdots\
\redsolapproxSTArg{\nbasisst} ]^T\in\RR{\nbasisst}$ and
Eq.~\eqref{eq:spacetimeExpansion} relates $\redsolapproxST$ to
$\solapproxSTFunc$.

Necessary first-order optimality conditions for Problem
\eqref{eq:t-lspgReducedLargerSimplify} correspond to stationarity of the
objective function, i.e., the solution $\redsolapproxST(\param)$ satisfies
\begin{equation} \label{eq:lspgCondition}
(\resRed(\cdot;\redsolapproxST(\param);\param),
\frac{\partial \resRed}{\partial\solRedDummyScalarArg{i}}(\cdot;\redsolapproxST(\param);\param)
)_{\weightmatst^T\weightmatst} = 0,\quad i\innat{\nbasisst},
\end{equation} 
where
 \begin{align} 
 \begin{split} 
\frac{\partial \resRed}{\partial\solRedDummyScalarArg{i}}
(t^n;\redsolapproxST;\param)
&= \frac{\partial \res}{\partial\solDummy}
\Bigl(t^n;
\solref +
\sum_{\ell=1}^{\nbasisst}
\stbasisvecFunc{\ell}
(\cdot)\redsolapproxSTArgNone{\ell};\param\Bigr)
\stbasisvecFunc{i}(\cdot)\\
&=
\sum_{j=0}^{\kn}\frac{\partial \resn}{\partial \solDummy^j}\Bigl(
\solref +
\sum_{\ell=1}^{\nbasisst}
\stbasisvecFunc{\ell}(t^n)\redsolapproxSTArg{\ell} 
,\ldots,
\solref +
\sum_{\ell=1}^{\nbasisst}
\stbasisvecFunc{\ell}(t^{n-\kn})\redsolapproxSTArg{\ell} 
;\param\Bigr)
\stbasisvecFunc{i}(t^{n-j})
\\
&=
\sum_{j=0}^{\kn}\left[
\alpha_j^n
-\dtArg{n} \beta_j^n\frac{\partial\flux}{\partial\solDummy}\Bigl(
\solref +
\sum_{\ell=1}^{\nbasisst}
\stbasisvecFunc{\ell}(t^{n-j})\redsolapproxSTArg{\ell}
,t^{n-j};\param\Bigr)\right]\stbasisvecFunc{i}(t^{n-j})
 \end{split} 
 \end{align} 
 denotes the elements of the test basis.
We refer to this approach
as a space--time least-squares Petrov--Galerkin (ST-LSPG) projection
because Eq.~\eqref{eq:lspgCondition} corresponds to a Petrov--Galerkin
projection with test basis $\{\frac{\partial \resRed}{\partial\solRedDummyScalarArg{i}}
(\cdot;\redsolapproxST;\param)
\}_{i\innat{\nbasisst}}$ and also satisfies
necessary conditions for the nonlinear least-squares problem
\eqref{eq:t-lspgReducedLargerFull}--\eqref{eq:t-lspgReducedLargerSimplify}.

As Problem
\eqref{eq:t-lspgReducedLargerFull}--\eqref{eq:t-lspgReducedLargerSimplify} is
simply a nonlinear least-squares
problem with $\weightmatstRows$ equations in $\nbasisst
(\leq \weightmatstRows)$ unknowns, we can solve it with the Gauss--Newton method, which leads to the following
sequence of iterates for $k=0,\ldots,\kmax(\param)-1$ given an initial guess $\redsolapproxSTIt{0}$:
\begin{align}\label{eqref:STGaussNewtonOne}
\left(\frac{\partial \resRed}{\partial\solRedDummy}(\cdot;\redsolapproxSTIt{k};\param)\deltaredsolapproxSTIt{k},
\frac{\partial \resRed}{\partial\solRedDummy}(\cdot;\redsolapproxSTIt{k};\param)
\right)_{\weightmatst^T\weightmatst}
&=
-\left(\resRed(\cdot;\redsolapproxSTIt{k};\param),
\frac{\partial \resRed}{\partial\solRedDummy}(\cdot;\redsolapproxSTIt{k};\param)
\right)_{\weightmatst^T\weightmatst} \\
\label{eqref:STGaussNewtonTwo}
\redsolapproxSTIt{k+1}&=\redsolapproxSTIt{k}+\lineSearchParamArg{k}\deltaredsolapproxSTIt{k},
\end{align}
where we set
$
\redsolapproxST(\param) =\redsolapproxSTIt{\kmax(\param)}
$
at convergence 
 and $\lineSearchParamArg{k}\in\RR{}$ denotes a step length that can be computed to
ensure global convergence (e.g., satisfy the strong Wolfe conditions).
Again, we note that the condition $\nbasisst\leq\weightmatstRows$ is necessary for
singular values of the Jacobian
$
\weightmatst\frac{\partial \resRed}{\partial\solRedDummy}(\cdot;\redsolapproxSTIt{k};\param)\in\RR{\weightmatstRows\times\nbasisst}
$
to be uniformly bounded away from zero in the region of interest; this is one
of the sufficient conditions required to prove convergence of the
Gauss--Newton method (see, e.g., \cite[Theorem 10.1]{NocedalWright}).

\begin{remark}[Simplification for block-diagonal weighting matrices]
	%\KTC{Fix residual arguments here!}
If the weighting matrix is block-diagonal, i.e.,
$\weightmatst=\diag{\weightmatstArg{n}}$ with
$\weightmatstArg{n}\in\RR{\weightmatstRowsn\times\ndof}$,
$\weightmatstRowsn\leq\ndof$, and $\weightmatstRows =
\sum_{n=1}^\ntimedof\weightmatstRowsn$, then Problems
\eqref{eq:t-lspgReducedLargerFull} and \eqref{eq:t-lspgReducedLargerSimplify}
simplify
to
\begin{equation} \label{eq:t-lspgReducedFull}
 \solapproxSTFuncOnlyParam =
	\underset{
\solDummy\in\spacetimesubspace
		}{\arg\min}
    \sum_{n=1}^\ntimedof\left \|\weightmatstArg{n}
		\resn(
\solDummy
;\param)
		\right \|_2^2\quad\text{and}\quad
 \redsolapproxST(\param) =
	\underset{\solRedDummyOpt\in\RR{\nbasisst}}{\arg\min}
    \sum_{n=1}^{\ntimedof}\left \|\weightmatstArg{n}
\resRed(t^n;\solRedDummyOpt;\param)
		\right \|_2^2,
\end{equation}
respectively.
%\begin{equation} \label{eq:t-lspgReduced}
% \redsolapproxST(\param) =
%	\underset{(\solDummyScalar_{ij})}{\arg\min}
%    \sum_{n=1}^{\ntimedof}\left \|\weightmatst^n
%		\resn(
%\solref +
%\sum_{i=1}^{\nbasisspace}\sum_{j=1}^{\nbasistime^i}
%\basisvecspacei{i}
%\basisvectimeFuncij{i}{j}(t^n)\solDummyScalar_{ij},
%		\ldots,
%		\solref +
%\sum_{i=1}^{\nbasisspace}\sum_{j=1}^{\nbasistime^i}
%\basisvecspacei{i}
%\basisvectimeFuncij{i}{j}(t^{n-\kn })\solDummyScalar_{ij};\param)
%		\right \|_2^2.
%\end{equation}
\end{remark}

\begin{remark}[Space--time Galerkin projection]
Using the present formalism, we can also derive a (discrete) space--time
Galerkin projection by enforcing Galerkin orthogonality rather than the
Petrov--Galerkin orthogonality in Eq.~\eqref{eq:lspgCondition}. In particular,
this space--time Galerkin method computes the solution
$\redsolapproxST_\text{G}(\param)$ satisfying \begin{equation}
\label{eq:galerkinCondition} (\resRed(\cdot;\redsolapproxST_\text{G};\param),
\stbasisvecFunc{i}(\cdot))_{\metric} = 0,\quad
i\innat{\nbasisst}\end{equation} for some
prescribed metric $\metric\in\SPSD{\ndof\ntimedof}$. However, because the
Galerkin solution $\redsolapproxST_\text{G}(\param)$ does not associate with
the solution to any optimization problem in general, we do not pursue this
method further. We note that this approach is the discrete
counterpart to the continuous Galerkin projection proposed in
Refs.~\cite{volkwein2006algorithm,baumann2016space}; however, these
contributions effectively employ $\metric=\identity{\ndof\ntimedof}$ and thus provide no
mechanism for hyper-reduction as we do in Section \ref{sec:spacetimehyper}.
\end{remark}

\subsection{Weighting matrix and hyper-reduction}\label{sec:spacetimehyper}

Hyper-reduction refers to reducing the computational complexity of
nonlinear ROMs by introducing approximations of the
nonlinear functions. This is typically achieved via collocation (wherein the
nonlinear functions are simply sampled)
\cite{astrid2007mpe,ryckelynck2005phm,LeGresleyThesis} or function-reconstruction
approaches (e.g., gappy POD \cite{sirovichOrigGappy}, empirical interpolation
\cite{barrault2004eim,chaturantabut2010journal}), wherein the nonlinear
function is approximated from a sample of its entries via interpolation or
least-squares regression
\cite{astrid2007mpe,bos2004als,chaturantabut2010journal,galbally2009non,drohmannEOI,CarlbergGappy,carlbergJCP}. In the case of 
(spatial-projection-based) LSPG, it has been shown  that
hyper-reduction can be realized by particular choices of the weighting matrix \cite{carlbergGalDiscOpt}.
We now propose several choices for the space--time weighting matrix
$\weightmatst$, some of which will lead to hyper-reduction for the
\methodAcronym\ method. 
\reviewerBRtwo{
All of these hyper-reduction methods will ensure that the computational cost of
	the \methodAcronym\ method is independent
	of both the spatial and temporal dimensions characterizing the FOM
	if the Jacobian of the space--time residual is sparse, i.e., the
	number of nonzeros in 
$
\vectorize{
	\frac{\partial\resEntryi{i}}{\partial \solDummy}(\timeArg{n};\solDummyOpt;\param)
}\in\RR{\nspacedof\ntimedof}
$, $i\innat{\nspacedof}$, $n\innat{\ntimedof}$
	is `small' and is independent of both the spatial and temporal dimensions
$\nspacedof$ and $\ntimedof$, respectively. This is sometimes referred to as
the $H$-independence condition \cite{drohmannEOI} and is inherited directly
from $H$-independence of the spatial residual $\resn$, as 
$
\nnz{\vectorize{
	\frac{\partial\resEntryi{i}}{\partial \solDummy}(\timeArg{n};\solDummyOpt;\param)
}} = 
\sum_{j=0}^{\knArg{n}}\nnz{
	\frac{\partial \resArgEntry{i}{n}}{\partial
\solDummy^{j}}(\solDummyOpt^{n},\ldots,\solDummyOpt^{n-\knArg{n}};\param)}
$, where $\nnzNo$ denotes the number of nonzeros of its argument.
}

\subsubsection{Unweighted LSPG}\label{sec:unweightedLSPG}

The most obvious choice for the weighting matrix is $\weightmatst =
\identity{\ndof\ntimedof}$; this choice simply minimizes the sum of squares of
elements of the residual over both space and time and leads to
$\weightmatstRows=\ndof\ntimedof$
with 
$\weightmatstArg{n}=\identity{\ndof}$, $n\innat{\ntimedof}$ in Problem
\eqref{eq:t-lspgReducedFull}. This is analogous
to unweighted LSPG in the spatial-projection case, in which case
$\weightmat=\identity{\ndof}$ in Problem \eqref{eq:lspggnat}--\eqref{eq:lspggnatReduced}. 

However, because this approach requires evaluating all $\ndof\ntimedof$ of the
space--time residual in order to compute the objective function, it precludes
significant computational-cost savings. As was also pointed out in
Ref.~\cite{baumann2016space}, this bottleneck is especially
cumbersome for space--time ROM approaches. For this reason, alternative choices for the
weighting matrix $\weightmatst$ can be employed that lead to a ROM whose
computational complexity is independent
of the full spatiotemporal dimension $\ndof\ntimedof$.  We now describe
different choices for the weighting matrix $\weightmatst$ that lead to such
hyper-reduction.

\subsubsection{Space--time collocation}\label{sec:spacetimecoll}
 We can extend collocation hyper-reduction to the \methodAcronym\ context by employing
 weighting matrix $\weightmatst = \samplematst$ with
\begin{equation}\label{eq:samplematst}
 \samplematst
 \defeq\left[\unitvecArg{\spacetimeindexIncludeAll{1}}\ \cdots\
 \unitvecArg{\spacetimeindexIncludeAll{\nressamplest}}\right]^T
 \in\{0,1\}^{\nressamplest\times\nspacedof\ntimedof},
 \end{equation}
 where
 $\spacetimeindexNo:(i,j)\mapsto i + \ndof(j-1)$, and $\unitvecArg{i}$ denotes the $i$th canonical unit vector. Here,
 $\spacetimesampleset\defeq\{(\spaceindex{i},\timeindex{i})\}_{i\innat\nressamplest}\subseteq\nat{\nspacedof}\times\nat{\ntimedof}$
 %$\spaceindexNo:\nat{\nressamplest}\rightarrow\nat{\ndof}$ and
 %$\timeindexNo:\nat{\nressamplest}\rightarrow\nat{\ntimedof}$ 
 denotes the set
 of space--time sample indices. Again, we require 
 $\nbasisst\leq\nressamplest(\leq\nspacedof\ntimedof)$ to ensure nonsingular
 residual Jacobians. 
 
 Critically, note that applying $\weightmatst = \samplematst$ in Problem
 \eqref{eq:t-lspgReducedLargerFull}--\eqref{eq:t-lspgReducedLargerSimplify} leads to hyper-reduction, as evaluating the
 objective function requires evaluating only $\nressamplest <
 \ndof\ntimedof$ elements of the spatiotemporal residual. In practice, this
 implies that the residual will be evaluated only at a subset of time
 instances and spatial degrees of freedom. Further, this can also lead to a positive
 semidefinite metric $\metric = \weightmatst^T\weightmatst$, as
 $\rank(\weightmatst^T\weightmatst) = \nressamplest\leq\ndof\ntimedof$ in this
 case.

\subsubsection{Space--time GNAT}\label{sec:spacetimegnat}

 Similarly, we can extend gappy POD hyper-reduction to \methodAcronym\ by employing a
 weighting matrix $\weightmatst =
 (\samplematst\basismatrest)^+\samplematst\in\RR{\nbasisresst\times\nspacedof\ntimedof}$, where $\basismatrest
 \in \orthomat{\nspacedof\ntimedof }{ \nbasisresst}$ denotes a basis (in
 matrix form) for the spatiotemporal residual. Again, we require
 $\nbasisst\leq\nbasisresst\leq\nressamplest (\leq\nspacedof\ntimedof)$ for nonsingular
 residual Jacobians.
 When 
 this weighting matrix is employed, Problem
 \eqref{eq:t-lspgReducedLargerFull}--\eqref{eq:t-lspgReducedLargerSimplify} is equivalent to minimizing the
 $\ell^2$-norm of the gappy POD-approximated residual, i.e.,  Problem
 \eqref{eq:t-lspgReducedLargerFull}--\eqref{eq:t-lspgReducedLargerSimplify} with $\weightmatst =
 (\samplematst\basismatrest)^+\samplematst\in\RR{\nbasisresst\times\nspacedof\ntimedof}$
 is equivalent to
\begin{equation} \label{eq:t-lspgReducedLargerSimplifyGappy}
\solapproxSTFuncOnlyParam =
	\underset{
\solDummyOpt\in\spacetimesubspace
		}{\arg\min}
    \left \|
		\resstApprox(\cdot;
\solDummyOpt;\param)
		\right \|_2^2
\quad\text{and}\quad
 \redsolapproxST(\param) =
\underset{\solRedDummyOpt\in\RR{\nbasisst}}{\arg\min}
    \left \|\resRedApprox(\cdot;\solRedDummyOpt;\param)
		\right \|_{2}^2,
		%=
	%\underset{\solRedDummyOpt\in\RR{\sum_{i=1}^\nbasisspace \nbasistime^i}}{\arg\min}
    %\left \|
%\resRedstapprox(\solRedDummyOpt;\param)
		%\right \|_2^2,
\end{equation}
where we have defined the gappy POD residual approximations
as 
\begin{align}
\begin{split}
\vectorize{\resstApprox(\cdot;\solDummyOpt;\param)}
&=
\underset{\bar\solDummyOpt\in\range{\basismatrest}}{\arg\min}\|\samplematst\bar\solDummyOpt
- \samplematst\resst(\solDummyOpt;\param)\|_2^2 =
\basismatrest(\samplematst\basismatrest)^+\samplematst\resst(\solDummyOpt;\param)\\
\vectorize{\resRedApprox(\cdot;\solRedDummyOpt;\param)}
&=
\underset{\bar\solDummyOpt\in\range{\basismatrest}}{\arg\min}\|\samplematst\bar\solDummyOpt
- \samplematst\resRedst(\solRedDummyOpt;\param)\|_2^2 =
\basismatrest(\samplematst\basismatrest)^+\samplematst\resRedst(\solRedDummyOpt;\param).
\end{split}
\end{align}
  
As with space--time collocation, space--time gappy POD leads to
hyper-reduction, as evaluating the objective function requires evaluating only
$\nressamplest \leq \ndof\ntimedof$ elements of the spatiotemporal residual.
It can also lead to a positive
 semidefinite metric $\metric = \weightmatst^T\weightmatst$, as
 $\rank(\weightmatst^T\weightmatst) = \nbasisresst\leq\ndof\ntimedof$ in this
 case.
Due to its close relationship to the original GNAT method---in
which case
$\weightmat = (\samplemat \basismatres)^{+}\samplemat\in\RR{\nbasisres\times\nspacedof}$
in Problem \eqref{eq:lspggnat}--\eqref{eq:lspggnatReduced}, we refer to this
approach as the space--time GNAT (ST-GNAT) method.

\section{Computing method ingredients}\label{sec:ingredients}

This section describes particular methods for constructing the ingredients
required for the \methodAcronym\ method, namely the space--time trial subspace
$\spacetimesubspaceFunc$; the sampling matrix
$\samplematst$ in the case of hyper-reduction; and the residual basis
$\basismatrest$ in the case of ST-GNAT.

\subsection{Space--time trial subspace}\label{sec:spacetimetrialconstruct}
We first assume that a set of training parameter instances
$\paramDomainTrain \defeq \{\paramTrain{1}, \ldots,
\paramTrain{\ntrain}\}\subset\paramDomain$
has been defined (by, e.g., uniform sampling, Latin-hypercube sampling, greedy
sampling) for which the FOM \eqref{eq:fom} has been solved numerically using a
linear multistep method \eqref{eq:LMMresidual} to obtain state `snapshots'
$\solFuncOnlyParamTrain{i}\in\RR{\nspacedof}$, $i\innat{\ntrain}$. This data-collection process is
referred to as the `offline' stage in model reduction. 

Previous work on space--time model reduction constructed the space--time trial
subspace simply as the span of these snapshots, i.e.,
\begin{equation}
\spacetimesubspaceFunc=
\solref\otimes
\onesFunc+
\Span{\solFuncOnlyParamTrain{i}-\solinit(\paramTrain{i})}_{i=1}^{\ntrain}
\subseteq\spatialSpace\otimes\timeDiscreteFuncSpace
\end{equation}
such that
$\stbasisvecFunc{i} =
\solFuncOnlyParamTrain{i}-\solinit(\paramTrain{i})\in\RR{\nspacedof}\otimes\timeDiscreteFuncSpace$,
$i\innat{\ntrain}$
\cite{urban2012new,urban2014improved,yanoPateraUrban,yanoReview,constantineResMin}.
Unfortunately, because this approach extracts only a single space--time basis
vector from each training simulation, it limits the dimension of the
space--time ROM to $\nbasisst = \ntrain$. In practical contexts where a single
training simulation may incur significant computational costs, this can
significantly limit the dimensionality (and resulting accuracy) of the
space--time ROM. Further, storage costs can be significant in
this case, as the basis requires $\nbasisst\nspacedof\ntimedof$ storage.

To overcome these shortcomings, we propose to compute the space--time trial
subspace by applying tensor-decomposition techniques to the
three-way `state
 tensor' $\tensor
 \in\RR{\nspacedof\times\ntimedof\times\ntrain}$ with
 elements 
\begin{equation}\label{eq:tensor}
\tensor_{ijk}\defeq
\solEntry{i}(t^j;\paramTrain{k})-\solinitEntry{i}(\paramTrain{k}),\quad
i\innat{\nspacedof},\ j\innat{\ntimedof},\ k\innat{\ntrain}.
\end{equation}
The resulting 
space--time trial
subspace comprises the direct sum of Kronecker products of spatial and
temporal subspaces, i.e.,
\begin{equation}
\spacetimesubspace=
\solref\otimes
\onesFunc+
\oplus_{i=1}^{\nbasisspace}
\spatialSubspacei{i}\otimes\temporalSubspaceFuncArg{i},
\end{equation}
where
$ \spatialSubspacei{i} \defeq \Span{\basisvecspacei{i}} \subseteq
\RR{\nspacedof}$, $\temporalSubspaceFuncArg{i} \defeq
\Span{\basisvectimeFuncij{i}{j}}_{j=1}^{\nbasistime^i} \subseteq \timeDiscreteFuncSpace$,
and $\nbasistime^i\leq \ntimedof$	(hopefully with $\nbasistime^i\ll \ntimedof$) for $i\innat{\nbasisspace}$. Here,
$\basisvectimeFuncij{i}{j} \in \timeDiscreteFuncSpace$ with
$\basisvectimeFuncij{i}{j}(0)=0$
denotes the $j$th temporal basis
vector associated with the temporal behavior 
 of $i$th spatial basis vector $\basisvecspacei{i}$.
 Thus, in this case we have
 $\stbasisvecFunc{\mapping{i}{j}} = 
 \basisvecspacei{i}\basisvectimeFuncij{i}{j}\in\RR{\nspacedof}\otimes\timeDiscreteFuncSpace$, $i\innat{\nbasisspace}$, $j\innat{\nbasistime^i}$ and $\nbasisst =
 \sum_{i=1}^{\nbasisspace}\nbasistime^i$,
where
$\mappingNo:(i,j)\mapsto \sum_{k=1}^{i-1}\nbasistime^k + j $ 
provides a mapping from the spatial-basis and temporal-basis indices to a
space--time basis index.
This approach enables a larger space--time ROM dimension, as 
$\nbasisst =
 \sum_{i=1}^{\nbasisspace}\nbasistime^i>\ntrain$ is possible; further, this
 approach requires only 
 $\nbasisspace\nspacedof + \nbasisst\ntimedof$ storage at most. We expect the
resulting \methodAcronym\ ROM be accurate if the solution exhibits separable
behavior in space and time.

The mode-1 and mode-2 unfolding of $\tensor$ can be written as
\begin{align}
&\tensorUnfold{1}= \bmat{\snapshots(\paramTrain{1}) &
\ldots & \snapshots(\paramTrain{\ntrain})} \in
\RR{\nspacedof\times\ntimedof\ntrain}\\
  &\tensorUnfold{2}=\bmat{\snapshots^T(\paramTrain{1}) &
	\ldots & \snapshots^T(\paramTrain{\ntrain})} \in
	\RR{\ntimedof\times\nspacedof\ntrain},
\end{align} 
respectively, where we have defined
   $\snapshots(\param) \defeq
	 \unrollfunc{\solFuncOnlyParam-\solinit(\param)}$.
In the model-reduction literature, the matrix $\snapshotmatspace$ is typically referred to as the
`global snapshot matrix'; we refer to \reviewerB{it in} this work as the `spatial snapshot
matrix', as its columns comprise snapshots of the spatial solution over time
and parameter variation. Similarly, we refer to $\snapshotmattime$ as the
`temporal snapshot matrix', as its columns comprise snapshots of the
time-evolution of the solution over variation in space and parameter.
 
 \subsubsection{Spatial subspaces}\label{sec:constructSpatialSubspaces}
 We propose to compute the spatial subspaces $\spatialSubspacei{i} \defeq
 \Span{\basisvecspacei{i}}$, $i\innat{\nbasisspace}$ via proper orthogonal
 decomposition (POD) applied to the spatial snapshot matrix
 $\snapshotmatspace$. In particular, we compute the spatial bases
 from the  singular value decomposition (SVD) as
 \begin{align} \label{eq:spatialSVD}
  \tensorUnfold{1} &= \leftsingmatspace \singvalmatspace
	\rightsingmatspace^T\in\RR{\ndof\times\ntimedof\ntrain}\\
	\basisvecspacei{i}& = \leftsingvecspacei{i},\quad i\innat\nbasisspace.
 \end{align} 
 where $\nbasisspace\leq\min(\ndof,\ntimedof\ntrain)$ and 
 $\leftsingmatspace\equiv\left[\leftsingvecspacei{1}\ \cdots\
 \leftsingvecspacei{\ntimedof\ntrain}\right]$.
 The spatial subspace requires $\nbasisspace\nspacedof$  storage.
We now describe three approaches to computing the temporal subspaces
$\temporalSubspaceFuncArg{i}$, $i\innat{\nbasisspace}$ from the state tensor
$\tensor$.

\subsubsection{Fixed temporal
subspace via T-HOSVD}\label{sec:solutionSubspaceFixed}
The most straightforward approach is to compute a fixed temporal subspace
$\temporalSubspaceFunc\defeq\Span{\basisvectimeFunci{i}}_{i=1}^{\nbasistime}\subset \timeDiscreteFuncSpace$ such that
$\temporalSubspaceFuncArg{i}=\temporalSubspaceFunc$ and
$\nbasistime^i=\nbasistime$,
$i\innat{\nbasisspace}$. We can compute such a subspace by applying 
POD to the
temporal snapshot matrix $\snapshotmattime$. Specifically, we can 
compute the fixed temporal bases
 from the SVD of $\snapshotmattime$ as
 \begin{align} \label{eq:temporalSVD2}
  \snapshotmattime &= \leftsingmattime \singvalmattime
	\rightsingmattime^T\in\RR{\ntimedof\times\nspacedof\ntrain}
	\\
	\basisvectimei{j} &= \leftsingvectimeArg{j},\quad j\innat\nbasistime,
 \end{align} 
 where $\nbasistime\leq \ndof\ntrain$,
  $\basisvectimei{j}\defeq
\vectorize{\basisvectimeFunci{j}}$, and
$\leftsingmattime\equiv
\left[
\leftsingvectimeArg{1}
\ \cdots\ 
\leftsingvectimeArg{\nspacedof\ntrain}
\right]$.  This approach is equivalent to applying the truncated higher-order
SVD(T-HOSVD) \cite{de2000multilinear, tucker1966some} to the
state tensor $\tensor$, and it requires
$\nbasistime\ntimedof$ storage for the temporal subspace.
We note that this approach is similar to that
proposed in Ref.~\cite{baumann2016space} in the context of space--time
Galerkin projection performed at the time-continuous level.

\subsubsection{Fixed temporal
subspace via ST-HOSVD}\label{sec:solutionSubspaceFixedSTHOSVD}
Alternatively, we recall from Eq.~\eqref{eq:spacetimeLSPGcontinuous} that each
space--time basis vector is the
Kronecker product of a spatial basis vector with a temporal basis vector.
Thus, it is sensible to compute the temporal subspace according to the
observed time evolution of the solution in the coordinates defined by the
spatial subspace. Mathematically, we can achieve this by applying the
sequentially truncated HOSVD (ST-HOSVD)
\cite{vannieuwenhoven2012new,austin2015parallel}.
Rather than computing the SVD of $\snapshotmattime$, which is agnostic to
dimensionality reduction in space, ST-HOSVD instead computes
the SVD of the mode-2 unfolding of $\tensorReduce{\basismatspace}$, where we have defined
$\tensorReduce{\basisDummy}\defeq\tensor\times_1 \basisDummy$ for
$\basisDummy\in\RR{\ndof\times p}$ such that
	$
	\tensorReduceUnfold{\basisDummy}{1} = \basisDummy^T\snapshotmatspace 	$.
 In
particular, we
have
\begin{align}\label{eq:temporalSVD3}
\tensorReduceUnfold{\basismatspace}{2} &= \leftsingmattime(\basismatspace)
\singvalmattime(\basismatspace)
\rightsingmattime(\basismatspace)^T\in\RR{\ntimedof\times\nbasisspace\ntrain}
	\\
	\label{eq:temporalSVD3Next}
\basisvectimei{j} &= \leftsingvectimeArg{j}(\basismatspace),\quad
	j\innat\nbasistime,
\end{align}
where $\nbasistime \leq \nbasisspace\ntrain$ and $
\leftsingmattime(\basismatspace)
\equiv\left[
\leftsingvectimeArg{1}(\basismatspace)
\cdots\
\leftsingvectimeArg{\nbasisspace\ntrain}(\basismatspace)
\right]$.
In addition to enabling the temporal basis to be associated with the time
evolution of the spatial basis
$\basismatspace$, this approach is less computationally expensive than
applying the T-HOSVD, as
$\tensorReduceUnfold{\basismatspace}{2}\in\RR{\ntimedof\times\nbasisspace\ntrain}$,
while $\tensorUnfold{2}\in \RR{\ntimedof\times\nspacedof\ntrain}$ and
	typically $\nbasisspace\ll\ndof$.
This temporal subspace also requires
$\nbasistime\ntimedof$ storage.

\subsubsection{Tailored temporal
subspaces via ST-HOSVD}\label{sec:solutionSubspaceTailored}
We can further tailor the temporal bases to capture the time evolution of
each individual spatial basis vector. To achieve this using the ST-HOSVD, we
compute the bases as
\begin{align}
\label{eq:tailoredSVD}\tensorReduceUnfold{\basisvecspacei{i}}{2} &=
\leftsingmattime(\basisvecspacei{i}) \singvalmattime(\basisvecspacei{i})
\rightsingmattime(\basisvecspacei{i})^T\in\RR{\ntimedof\times\ntrain}
	\\
	\label{eq:tailoredSVD2}\basisvectimeij{i}{j} &= \leftsingvectimeArg{j}(\basisvecspacei{i}),\quad
	i\innat{\nbasisspace},\ j\innat{\nbasistime^i},
\end{align}
where $\basisvectimeij{i}{j}\defeq
\vectorize{\basisvectimeFuncij{i}{j}}$ and $
\leftsingmattime(\basisvecspacei{i}) 
\defeq\left[
\leftsingvectimeArg{1}(\basisvecspacei{i})\ \cdots\
\leftsingvectimeArg{\ntrain}(\basisvecspacei{i})
\right]$.
This approach generates a tailored temporal subspace
$\temporalSubspaceFuncArg{i} \defeq
\Span{\basisvectimeFuncij{i}{j}}_{j=1}^{\nbasistime^i}$ for each spatial
subspace $ \spatialSubspacei{i} \defeq \Span{\basisvecspacei{i}}$. Further,
because
$\tensorReduceUnfold{\basisvecspacei{i}}{2}\in\RR{\ntimedof\times\ntrain}$,
$i\innat{\nbasisspace}$, the cost of computing the $\nbasisspace$ SVDs \eqref{eq:tailoredSVD} is significantly less than
computing the SVDs in either \eqref{eq:temporalSVD2}
or \eqref{eq:temporalSVD3}; this results from the quadratic dependence of the
SVD cost on the number of columns in the matrix. However, the maximum
dimension of each temporal basis is limited to the number of
training-parameter instances, i.e., $\nbasistime^i\leq\ntrain$,
$i\innat{\nbasisspace}$. 
This temporal subspace requires
$\sum_{i=1}^\nbasisspace\nbasistime^i\ntimedof=\nbasisst\ntimedof$ storage,
which is larger than that required by the fixed temporal subspaces.  We note
that this was the approach employed to construct temporal bases in our
previous work based on forecasting
\cite{carlberg2015decreasing,carlberg2016data}.

\subsection{Space--time residual basis}\label{sec:spacetimeresidual}

We propose to construct the space--time residual basis $\basismatrest\in
\orthomat{\nspacedof\ntimedof }{ \nbasisresst}$ from training data comprising
the space--time residual computed at a set of pairs of reduced coordinates and
parameter instances, i.e.,
$\{\redsolapproxTrainArg{i},\paramTrainres{i}\}_{i\innat{\nrestrain}}$.  In
this case, the space--time residual `snapshots' can be expressed in a residual
tensor 
$\tensorres
 \in\RR{\nspacedof\times\ntimedof\times \nrestrain}$ with entries
\begin{equation}\label{eq:tensorres}
\tensorres_{ijk}\defeq
\resRedEntryi{i}(t^j;\redsolapproxTrainArg{k};\paramTrainres{k}).
\end{equation}
We propose three methods for determining these training instances
$\{\redsolapproxTrainArg{i},\paramTrainres{i}\}_{i\innat{\nrestrain}}$.
\begin{enumerate} 
\item\label{resTrain:ROMtraining} \textit{\methodAcronym\ ROM training iterations}. This
approach employs
\begin{equation}
\{\redsolapproxTrainArg{i},\paramTrainres{i}\}_{i\innat{\nrestrain}} = 
\{\redsolapproxSTIt{k}(\param), \param\}_{\param\in\paramDomainTrainres,\
k\innatZero{\kmax(\param)}},
\end{equation}
where $\redsolapproxSTIt{k}(\param)$ corresponds to the \methodAcronym\ solution
at the $k$th Gauss--Newton iteration
\eqref{eqref:STGaussNewtonOne}--\eqref{eqref:STGaussNewtonTwo} for some
specified weighting matrix $\weightmatst$ that does not rely on data (e.g.,
$\weightmatst = \identity{\ndof\ntimedof}$), and $\paramDomainTrainres \subset\paramDomain$ denotes a set of training
parameter instances that is in general different from $\paramDomainTrain$.
This case leads to $\nrestrain =
\sum_{\param\in\paramDomainTrainres}(\kmax(\param)+1)$ and requires
$\card{\paramDomainTrainres}$ training simulations of the \methodAcronym\ ROM.
\item \textit{Projection of FOM training solutions}. This approach employs
\begin{equation}
\{\redsolapproxTrainArg{i},\paramTrainres{i}\}_{i\innat{\nrestrain}} = 
\{\solFuncOnlyParamProject, \param\}_{\param\in\paramDomainTrainres},
\end{equation}
where $\solFuncOnlyParamProject$ is defined as
\begin{equation} \label{eq:FOMprojection}
\solFuncOnlyParamProject\defeq
(\stbasismat^T\stbasismat)^{-1}\stbasismat^T\vectorize{\solFuncOnlyParam-\solinit(\param)}
,\end{equation} 
 where $\stbasismat\defeq\bmat{
\stbasisvec{1} & \cdots &\stbasisvec{\nbasisst}}
\in\RR{\nspacedof\ntimedof\times \nbasisst}
$.
This approach does not require any additional
training simulations; it simply requires $\nrestrain$
evaluations of the space--time residual.
\item \textit{Random samples}. In this approach, the training set
$\{\redsolapproxTrainArg{i},\paramTrainres{i}\}_{i\innat{\nrestrain}}$
comprises random samples (e.g., via Latin hypercube sampling) from
$\sampleSpaceGenCoord\times\paramDomain$ where
$\sampleSpaceGenCoord\subseteq\RR{\nbasisst}$.
This approach also requires only $\nrestrain$ evaluations of the
space--time residual.
\end{enumerate}
Given the residual tensor $\tensorres$, we can compute the associated space--time residual basis in
a manner analogous to the approaches proposed in Section
\ref{sec:spacetimetrialconstruct}.
That is, we can compute spatial residual bases as 
 \begin{align} \label{eq:spatialSVDRes}
  \tensorresUnfold{1} &= \leftsingmatspaceres \singvalmatspaceres
	\rightsingmatspaceres^T\in\RR{\nspacedof\times\ntimedof \nrestrain}\\
	\basisvecspaceresi{i}& = \leftsingvecspaceresi{i},\quad i\innat\nbasisspaceres
 \end{align} 
 with $\nbasisspaceres\leq\ntimedof \nrestrain$
 and temporal residual bases either via the T-HOSVD
 \begin{align} \label{eq:temporalresSVD2}
  &\tensorresUnfold{2} = \leftsingmattimeres \singvalmattimeres
	\rightsingmattimeres^T\in\RR{\ntimedof\times \nspacedof \nrestrain}
	\\
	&\basisvectimeresi{j} = \leftsingvectimeresArg{j},\quad
	j\innat\nbasistimeres
 \end{align} 
 with $\nbasistimeres\leq\nspacedof \nrestrain$,
 the ST-HOSVD
\begin{align}
\label{eq:globalresSVD}&\tensorresReduceUnfold{\basismatspaceres}{2} =
\leftsingmattimeres(\basismatspaceres) \singvalmattimeres(\basismatspaceres)
\rightsingmattimeres(\basismatspaceres)^T\in\RR{\ntimedof\times\nbasisspaceres \nrestrain}
	\\
\label{eq:globalresSVD2}	&\basisvectimeresi{j} = \leftsingvectimeresArg{j}(\basismatspaceres),\quad
	j\innat{\nbasistimeres}
\end{align}
with $\nbasistimeres\leq\nbasisspaceres \nrestrain$,
 or the tailored ST-HOSVD
\begin{align}
\label{eq:tailoredresSVD}&\tensorresReduceUnfold{\basisvecspaceresi{i}}{2} =
\leftsingmattimeres(\basisvecspaceresi{i}) \singvalmattimeres(\basisvecspaceresi{i})
\rightsingmattimeres(\basisvecspaceresi{i})^T\in\RR{\ntimedof\times
\nrestrain},\quad
	i\innat{\nbasisspaceres}
	\\
\label{eq:tailoredresSVD2}	&\basisvectimeresij{i}{j} = \leftsingvectimeresArg{i}(\basisvecspaceresi{i}),\quad
	i\innat{\nbasisspaceres},\ j\innat{\nbasistimeres^i}
\end{align}
with $\nbasistimeres^i\leq\nrestrain$,
where 
$\tensorresReduce{\basisDummy}\defeq\tensorres\times_1 \basisDummy$,
$\basisvectimeresi{j}\defeq
\vectorize{\basisvectimeresFunci{j}}$,
and
$\basisvectimeresij{i}{j}\defeq
\vectorize{\basisvectimeresFuncij{i}{j}}$.

We compute the orthogonal residual basis $\basismatrest\in
\orthomat{\nspacedof\ntimedof }{ \nbasisresst}$ from the QR factorization of
the nonorthogonalized basis 
$\basisres\equiv\bmat{\basisresvecArg{1}&\cdots&\basisresvecArg{\nbasisresst}}\in\RR{\nspacedof\ntimedof\times\nbasisresst}
$ as
 \begin{equation} 
	\basisres
=\basismatrest\rfactor,
  \end{equation} 
where $\basisresvecij{i}{j} =
\vectorize{\basisvecspaceresi{i}\otimes\basisvectimeresFuncij{i}{j}}$ and
$\mappingresNo:(i,j)\mapsto \sum_{k=1}^{i-1}\nbasistimeres^k + j $ provides a
mapping from the spatial-basis and temporal-basis indices to a space--time
basis index for the residual.
\subsection{Sampling matrix}\label{sec:constructSamplingMatrix}

We propose three approaches for computing the space--time sample set
$\spacetimesampleset\defeq\{(\spaceindex{i},\timeindex{i})\}_{i\innat\nressamplest}$ that
defines the residual-sampling matrix $\samplematst$ in
Eq.~\eqref{eq:samplematst}.
\begin{enumerate} 
\item\label{greedySpaceTime} \textit{Greedy sampling of space--time indices}. This approach selects
space--time indices in a greedy manner by executing Algorithm
\ref{al:greedySpacetimeSample}, which is a
space--time
adaptation of the greedy method presented in
Ref.~\cite{CarlbergGappy,carlbergJCP} that allows for oversampling to 
enable least-squares regression via gappy POD.
\item\label{greedySpaceThenTime} \textit{Sequential greedy sampling of spatial then temporal indices}. This approach computes space--time sample indices as the
Cartesian product of spatial and temporal samples, i.e., $\spacetimesampleset =
\spacesampleset\times\timesampleset$. First, the approach selects spatial indices
$\spacesampleset$ by executing Algorithm \ref{al:greedySpatialSample} with
inputs $\basismatrest$, the desired number of spatial samples $\nresspaceind$, and $\timesampleset=\nat{\ntimedof}$
(i.e., full temporal sampling).
Then, the method selects temporal indices $\timesampleset$ by executing
Algorithm \ref{al:greedyTemporalSample} with inputs
$\basismatrest$, the desired number of temporal samples $\nrestimeind$, and $\spacesampleset$ computed from Algorithm
\ref{al:greedySpatialSample}.
\item\label{greedyTimeThenSpace} \textit{Sequential greedy sampling of temporal then spatial indices}.
This approach also computes space--time sample indices as $\spacetimesampleset =
\spacesampleset\times\timesampleset$. First, the approach selects temporal indices
$\timesampleset$ by executing Algorithm \ref{al:greedyTemporalSample} with
inputs $\basismatrest$, the desired number of temporal samples $\nrestimeind$, and $\spacesampleset=\nat{\nspacedof}$
(i.e., full spatial sampling).
Then, the method selects spatial indices $\spacesampleset$ by executing
Algorithm \ref{al:greedySpatialSample} with inputs
$\basismatrest$, 
the desired number of spatial samples
$\nresspaceind$, and $\timesampleset$ computed from Algorithm
\ref{al:greedyTemporalSample}.
\end{enumerate}
We note that enforcing $\spacetimesampleset =
\spacesampleset\times\timesampleset$ as in approaches
\ref{greedySpaceThenTime} and \ref{greedyTimeThenSpace} above comes with a
practical advantage. Namely, a single sample mesh
\cite{carlbergJCP}---which is tasked with computing spatial samples associated
with $\spacesampleset$---can be employed for all sampled time instances
$\timesampleset$.
  \begin{algorithm}[t]
 \caption{Greedy algorithm for constructing spatiotemporal sample set
 $\spacetimesampleset$}
 \label{al:greedySpacetimeSample}
 \textbf{Input:} residual basis
 $\basismatrest\in\orthomat{\nspacedof\ntimedof}{\nbasisresst}$; desired number of spatial samples
 $\nressamplest\leq\nspacedof\ntimedof$ \\ 
 \textbf{Output:} space--time sample set
 $\spacetimesampleset\subseteq\nat{\nspacedof}\times\nat{\ntimedof}$
   \begin{algorithmic}[1]
		 \STATE $\spacetimesampleset \leftarrow \emptyset$
		 \reviewerA{\COMMENT{Initialize spatiotemporal sample set.}}
			\STATE Determine number of spatiotemporal samples to compute at each greedy iteration:
			$$\spacetimesamplestoadd{i} =
			\begin{cases}\text{floor}(\nressamplest/\nbasisresst)+1,\quad
			i=1,\ldots,\nressamplest\ \text{mod}\ \nbasisresst\\
\text{floor}(\nressamplest/\nbasisresst),\quad
i=\nressamplest\ \text{mod}\ \nbasisresst+1,\ldots,\nbasisresst
			\end{cases}$$
      \FOR[Greedy iteration]{ $i=1,\ldots,\nbasisresst$ }
				\IF{$i=1$}
				\STATE $\errorvec \leftarrow \basismatrestvec{1}$ \reviewerA{\COMMENT{Initialize
				the error vector.}}
				\ELSE{}
				\STATE $\errorvec \leftarrow
				\left(\identity{\ourReReading{\nspacedof\ntimedof}} -
\left[\basisvecressti{1} \ \cdots \ \basisvecressti{i-1}\right]\left(\samplematst
\left[\basisvecressti{1} \ \cdots \
\basisvecressti{i-1}\right]\right)^+\samplematst\right)\basisvecressti{i} $, where
$\samplematst$ is defined in Eq.~\eqref{eq:samplematst}. \reviewerA{\COMMENT{Compute the
error in the gappy POD approximation of $\basisvecressti{i}$.}}
				\ENDIF
					\FOR{$j=1,\ldots,\spacetimesamplestoadd{i}$}
					\STATE
					$(\spaceindexNo^{\star},\timeindexNo^{\star})=\arg\max_{(k,n)\in\nat{\nspacedof}\times\nat{\ntimedof}\setminus\spacetimesampleset}|\errorvecFuncEntry{k}(t^n)|$,
					where $\errorvecFunc \equiv\left[\errorvecFuncEntry{1}\ \cdots\
					\errorvecFuncEntry{\nspacedof}\right]^T\defeq\vectorizeinv{\errorvec}$\\
					\reviewerA{\COMMENT{Identify the spatiotemporal index with the largest gappy POD
					error.}}
					\STATE $\spacetimesampleset \leftarrow \spacetimesampleset\cup
					\{(\spaceindexNo^{\star},\timeindexNo^{\star})\}$
					\reviewerA{\COMMENT{Include the identified space--time index in the
					spatiotemporal sample set.}}
					\ENDFOR
      \ENDFOR
   \end{algorithmic}
 \end{algorithm}
  \begin{algorithm}[t]
 \caption{Greedy algorithm for constructing temporal sample set $\timesampleset$}
 \label{al:greedyTemporalSample}
 \textbf{Input:}
 residual basis
 $\basismatrest\in\orthomat{\nspacedof\ntimedof}{\nbasisresst}$; desired number of temporal samples
 $\nrestimeind\leq\ntimedof$; spatial sample set $\spacesampleset\subseteq\nat{\nspacedof}$\\ 
 \textbf{Output:} temporal sample set $\timesampleset\subseteq\nat{\ntimedof}$
   \begin{algorithmic}[1]
		 \STATE  $\timesampleset \leftarrow \emptyset$
		 \reviewerA{\COMMENT{Initialize temporal sample set.}}
			\STATE Determine number of temporal samples to compute at each greedy iteration:
			$$\timesamplestoadd{i} =
			\begin{cases}\text{floor}(\nrestimeind/\nbasisresst)+1,\quad
			i=1,\ldots,\nrestimeind\ \text{mod}\ \nbasisresst\\
\text{floor}(\nrestimeind/\nbasisresst),\quad
i=\nrestimeind\ \text{mod}\ \nbasisresst+1,\ldots,\nbasisresst
			\end{cases}$$
      \FOR[Greedy iteration]{ $i=1,\ldots,\nbasisresst$ }
				\IF{$i=1$}
				\STATE $\errorvec \leftarrow \basismatrestvec{1}$ \reviewerA{\COMMENT{Initialize
				the error vector.}}
				\ELSE{}
\STATE $\errorvec \leftarrow \left(\identity{\nspacedof\ntimedof} -
\bmat{\basisvecressti{1} & \cdots & \basisvecressti{i-1}}\left(\samplematst
\bmat{\basisvecressti{1} & \cdots & \basisvecressti{i-1}}\right)^+\samplematst\right)\basisvecressti{i} $, where
$\samplematst$ is defined in Eq.~\eqref{eq:samplematst} with
$\spacetimesampleset = \spacesampleset\times\timesampleset$.\\
\reviewerA{\COMMENT{Compute the
error in the gappy POD approximation of $\basisvecressti{i}$.}}
				\ENDIF
					\FOR{$j=1,\ldots,\timesamplestoadd{i}$}
					\STATE
					$\timeindexNo^\star=\arg\max_{n\innat\ntimedof\setminus\timesampleset}\|\errorvecFunc(t^n)\|_2^2$,
					where $\errorvecFunc \defeq\vectorizeinv{\errorvec}$\\
					\reviewerA{\COMMENT{Identify the temporal index with the largest gappy POD error averaged
					over all spatial indices.}}
					\STATE $\timesampleset \leftarrow \timesampleset\cup \{\timeindexNo^\star\}$
					\reviewerA{\COMMENT{Include the identified temporal index in the
					temporal sample set.}}
					\ENDFOR
      \ENDFOR
   \end{algorithmic}
 \end{algorithm}

  \begin{algorithm}[t]
 \caption{Greedy algorithm for constructing spatial sample set $\spacesampleset$}
 \label{al:greedySpatialSample}
 \textbf{Input:} residual basis
 $\basismatrest\in\orthomat{\nspacedof\ntimedof}{\nbasisresst}$; desired number of spatial samples
 $\nresspaceind\leq\nspacedof$; temporal sample set
 $\timesampleset\subseteq\nat{\ntimedof}$ \\ 
 \textbf{Output:} spatial sample set $\spacesampleset\subseteq\nat{\nspacedof}$
   \begin{algorithmic}[1]
		 \STATE $\spacesampleset \leftarrow \emptyset$
		 \reviewerA{\COMMENT{Initialize spatial sample set.}}
			\STATE Determine number of spatial samples to compute at each greedy iteration:
			$$\spacesamplestoadd{i} =
			\begin{cases}\text{floor}(\nresspaceind/\nbasisresst)+1,\quad
			i=1,\ldots,\nresspaceind\ \text{mod}\ \nbasisresst\\
\text{floor}(\nresspaceind/\nbasisresst),\quad
i=\nresspaceind\ \text{mod}\ \nbasisresst+1,\ldots,\nbasisresst.
			\end{cases}$$
      \FOR[Greedy iteration]{ $i=1,\ldots,\nbasisresst$ }
				\IF{$i=1$}
				\STATE $\errorvec \leftarrow \basismatrestvec{1}$ \reviewerA{\COMMENT{Initialize
				the error vector.}}
				\ELSE{}
\STATE $\errorvec \leftarrow \left(\identity{\ourReReading{\nspacedof\ntimedof}} -
\left[\basisvecressti{1} \ \cdots \ \basisvecressti{i-1}\right]\left(\samplematst
\left[\basisvecressti{1} \ \cdots \
\basisvecressti{i-1}\right]\right)^+\samplematst\right)\basisvecressti{i} $, where
$\samplematst$ is defined in Eq.~\eqref{eq:samplematst} with
$\spacetimesampleset = \spacesampleset\times\timesampleset$. \reviewerA{\COMMENT{Compute the
error in the gappy POD approximation of $\basisvecressti{i}$.}}
				\ENDIF
					\FOR{$j=1,\ldots,\spacesamplestoadd{i}$}
					\STATE
					$\spaceindexNo^\star=\arg\max_{k\innat\nspacedof\setminus\spacesampleset}\sum_{n=1}^\ntimedof\ourReReading{(}\errorvecFuncEntry{k}(t^n)\ourReReading{)}^2$,
					where $\errorvecFunc \equiv\left[\errorvecFuncEntry{1}\ \cdots\
					\errorvecFuncEntry{\nspacedof}\right]^T\defeq\vectorizeinv{\errorvec}$\\
					\reviewerA{\COMMENT{Identify the spatial index with the largest gappy POD
					error averaged over all temporal indices.}}
					\STATE $\spacesampleset \leftarrow \spacesampleset\cup
					\{\spaceindexNo^\star\}$
					\reviewerA{\COMMENT{Include the identified spatial index in the
					spatial sample set.}}
					\ENDFOR
      \ENDFOR
   \end{algorithmic}
 \end{algorithm}

\subsection{Initial guess}\label{sec:initialGuess}

One practical challenge of \methodAcronym\ relative to
(spatial-projection-based) LSPG is devising an accurate
initial guess $\redsolapproxSTIt{0}$ for the Gauss--Newton iterations
\eqref{eqref:STGaussNewtonOne}--\eqref{eqref:STGaussNewtonTwo}. In
LSPG, the initial guess employed when solving Problem
\eqref{eq:lspggnat}--\eqref{eq:lspggnatReduced} at a given time instance $t^n$
using the Gauss--Newton method can be
set to the solution from the previous time instance, i.e., 
$\redsolapproxArg{n(0)} = \redsolapproxArg{n-1}$. This choice typically leads
to rapid convergence due to the fact that the state undergoes limited
variation between time instances, particularly for small time steps $\dtArg{n}$.
Alternatively,
accurate initial guesses based on polynomial extrapolation or forecasting
\cite{carlberg2015decreasing} can be employed to further improve convergence.

However, in \methodAcronym, deriving an accurate initial guess
$\redsolapproxSTIt{0}(\param)$
is less straightforward. We propose computing
$\redsolapproxSTIt{0}(\param)$ as an interpolant of the generalized coordinates
$\redsolapproxST(\param)$ in the parameter space. That is, given the training
parameter instances $\paramDomainTrain\subset\paramDomain$ for which the FOM has been solved, we can compute the
projection of these solutions onto the space--time trial subspace as
$
\{\solFuncOnlyParamProject\}_{\param\in\paramDomainTrain}
$ with
%\begin{equation} 
%\solFuncOnlyParamProjectArg{i}{j}\defeq(\solFuncOnlyParam,\basisvecspacei{i}\basisvectimeFuncij{i}{j}(\cdot))_{\identity{\ndof\ntimedof}}.
%\end{equation} 
$\solFuncOnlyParamProject$ defined in Eq.~\eqref{eq:FOMprojection}.
Then, we can compute $\redsolapproxSTIt{0}(\param)$ via interpolation (or
least-squares regression) in the parameter space $\paramDomain$ using 
data $
\{\solFuncOnlyParamProject\}_{\param\in\paramDomainTrain}.
$ 

\section{Error analysis}\label{sec:error}

For simplicity, this section omits parameter dependence of all
operators \reviewerA{and assumes a uniform time step, i.e., $\dtArg{n} = \dt$,
$n\innat{\ntimedof}$}. Thus,
the FOM solution $\sol$ and the
\methodAcronym\ ROM solution $\solapproxST$ satisfy
\begin{align}
\label{eq:optLSPGST}
&\res(\cdot; \sol) = \zerobold\quad\text{and}
\quad \solapproxST = \underset{\solDummy\in
\spacetimesubspace
}{\arg\min} \|\res(\cdot; \solDummy)\|_{\weightmatst^T\weightmatst},
\end{align}
respectively.
We begin by stating assumptions that will be leveraged in subsequent analyses:
\begin{Assumption}[resume=assumption]
\item \label{assLipschitz}
There exists a constant $\lipschitzFK>0$ such that
	$$
\|\flux(\wbold,t)-\flux(\ybold,t)\|_2\leq\lipschitzFK
\|\wbold-\ybold\|_2,\quad\forall \wbold,\ybold\in\RR{\ndof},\ \forall
t\in[0,\totaltime]
	$$
\item 
\label{asssmallDt} The time step $\dt$ is sufficiently small such that
$$
\dt <
	\frac{\sigma_\text{min}(\Alm)}{\lipschitzFK\sigma_\text{max}(\Blm)},
$$
where
 \begin{equation*} 
   \Alm \defeq \bmat{\alpha_0^1 \identity{} &              &
	 &             
	           \\ \alpha_1^2 \identity{} & \alpha_0^2 \identity{} &
						 &             
		         \\  \ddots&  & \ddots
						 &                       &  
	           \\                 & \alpha_{\knArg{\ntimedof}}^{\ntimedof} \identity{} & \cdots & \alpha_0^{\ntimedof} \identity{}},
   \quad
   \Blm \defeq \bmat{\beta_0^1 \identity{} &              &
	 &             
	           \\ \beta_1^2 \identity{} & \beta_0^2 \identity{} &
						 &             
		         \\  \ddots&  & \ddots
						 &                         
	           \\                 & \beta_{\knArg{\ntimedof}}^{\ntimedof} \identity{} & \cdots & \beta_0^{\ntimedof} \identity{}},
 \end{equation*}  
 where $\identity{} = \identity{\nspacedof}$ here 
 and 
 $\sigma_\text{min}(\boldsymbol A)$ 
 and 
 $\sigma_\text{max}(\boldsymbol A)$ 
 denote the minimum and maximum singular values of the matrix $\boldsymbol
 A$, respectively.
\item \label{assBounded} The space--time weighting matrix $\weightmatst$ is defined such
that the residual in the weighted space--time norm is uniformly
bounded below by the $\ell^2$-norm of the residual over all elements of the
space--time trial subspace, i.e., there exists $\normEquiv>0$ such that $$
\|\weightmatst\res(\cdot;\solDummy)\|_{2}\geq\normEquiv\|\res(\cdot;\solDummy)\|_2,\quad\forall\solDummy\in\spacetimesubspace.
$$
\end{Assumption}

	\begin{lemma}\label{lem:reslipschitz}
	Under Assumption \ref{assLipschitz}, the linear multistep residual is also
	Lipschitz continuous, i.e., 
\begin{equation} \label{eq:lipschitz}
\left\| \res(\cdot;
\wbold)-
		\res(\cdot;
\ybold)
		\right\|_{\weightmatst^T\weightmatst}\leq
\lipschitzK\|\wbold-\ybold\|_2\quad\forall \wbold,\ybold\in\RR{\ndof}\otimes\timeDiscreteFuncSpace.
\end{equation} 
with Lipschitz constant
\begin{equation} \label{eq:deflipschitzK}
\lipschitzK \defeq \sigma_\text{max}(\weightmatst\Alm) +
\dt\lipschitzFK\sigma_\text{max}(\weightmatst\Blm).
\end{equation} 

	\end{lemma}
\begin{proof}
Defining $
\fluxst:\solDummy \mapsto
\vectorize{\flux(\solDummy(\cdot),\cdot)}
$ and
 noting that $\|\fluxst(\wbold)\|_2^2=
\sum_{n=1}^{\ntimedof}\|\flux(\wbold(t^n),t^n)\|_2^2\leq \lipschitzFK^2
\sum_{n=1}^{\ntimedof}\|\wbold(t^n)\|_2^2
=
\lipschitzFK^2\|\wbold\|_2^2
$, we have
\begin{equation} 
\|\fluxst(\wbold)\|_2\leq
\lipschitzFK\|\wbold\|_2,\quad
\forall\wbold\in\RR{\nspacedof}\otimes\timeDiscreteFuncSpace,
\end{equation} 
i.e., the Lipschitz constant of $\fluxst$ is identical to that of $\flux$.
Further noting that 
 \begin{equation}\label{eq:resLM}
\resst(\wbold) = \Alm\vectorize{\wbold} - \dt\Blm\fluxst(\wbold)
    + \vectorize{\rhsFunc}, 
 \end{equation} 
 where 
$\rhsFuncArg{n} = \alpha_n^n\solinit -
\dt\beta_n^n\flux(\solinit)$, we
 have from the triangle inequality
\begin{align}\label{eq:rLipschitz}
\begin{split}
\left\| \res(\cdot;
\wbold) - 
\res(\cdot;
\ybold)
		\right\|_{\weightmatst^T\weightmatst}&=
		\|\weightmatst(\resst(\wbold)-\resst(\ybold))\|_2\\
		&= 
		\|\weightmatst \Alm(\vectorize{\wbold}-\vectorize{\ybold}) -
		\dt\weightmatst
		\Blm(\fluxst(\wbold)-\fluxst(\ybold))
	\|_2 \\
	&\leq
	\left(\sigmamax{\weightmatst \Alm} +\dt\lipschitzFK\sigmamax{\weightmatst
	\Blm}\right)\|\wbold-\ybold\|_2,\quad \forall \wbold,\ybold\in\RR{\nspacedof}\times
	\timeDiscreteFuncSpace.
\end{split}
\end{align}
\end{proof}

	\begin{lemma}\label{lem:resinvlipschitz}
	Under Assumptions and \ref{assLipschitz} and \ref{asssmallDt}, 
	the linear multistep residual is also inverse Lipschitz continuous, i.e., 
\begin{equation} \label{eq:invlipschitz}
\left\| \res(\cdot;
\wbold)-
		\res(\cdot;
\ybold)
		\right\|_{2}\geq
\inverseLipschitzK\|\wbold-\ybold\|_2\quad\forall \wbold,\ybold\in\RR{\ndof}\otimes\timeDiscreteFuncSpace.
\end{equation} 
with inverse Lipschitz constant
\begin{equation} \label{eq:definvlipschitzK}
\inverseLipschitzK \defeq \sigma_\text{min}(\Alm) -
\dt\lipschitzFK\sigma_\text{max}(\Blm).
\end{equation} 
	\end{lemma}
\begin{proof}
Applying the reverse triangle inequality and employing Assumption \ref{asssmallDt}
yields
\begin{align}\label{eq:rLipschitzReverse}
\begin{split}
\left\| \res(\cdot;
\wbold) - 
\res(\cdot;
\ybold)
		\right\|_{2}&
=\|\Alm(\vectorize{\wbold}-\vectorize{\ybold})
-\dt\Blm(\fluxst(\wbold)-\fluxst(\ybold))
\|_2	\\
		&\geq
\| \Alm(\vectorize{\wbold}-\vectorize{\ybold})\|_2 - \dt\|
		\Blm(\fluxst(\wbold)-\fluxst(\ybold))
	\|_2,
\end{split}
\end{align}
which directly leads to the desired result.
\end{proof}

\begin{theorem}[\textit{a priori} error bound with respect to $\ell^2$-optimal solution]\label{thm:aprioritwo}
Under Assumptions \ref{assLipschitz}, \ref{asssmallDt},
and \ref{assBounded}, the error in the
\methodAcronym\ solution at any time instance can be bounded by the best
approximation error as
\begin{align}
\begin{split}
\|\sol-\solapproxST\|_{2}\leq
\frac{1}{\normEquiv}
\left(
\frac{\sigmamax{\weightmatst\Alm} + \dt\lipschitzFK\sigmamax{\weightmatst\Blm} }
{\sigmamin{\Alm}-\dt\lipschitzFK\sigmamax{\Blm}}
\right)
\underset{\solDummy\in
\spacetimesubspace
}{\min} \|\sol-\solDummy\|_2.
\end{split}
\end{align}
\end{theorem}
\begin{proof}
We begin by defining
the $\ell^2$-optimal solution
$\soltwo$, which satisfies
\begin{align}
\label{eq:opttwo}
\soltwo = \underset{\solDummy\in
\spacetimesubspace
}{\arg\min} \|\sol-\solDummy\|_2,
\end{align}
Then, we can exploit the optimality properties of $\soltwo$ and $\solapproxST$
from Eqs.~\eqref{eq:opttwo} and \eqref{eq:optLSPGST}, respectively; Lipschitz
continuity (Lemma \ref{lem:reslipschitz}); residual-norm equivalence
(Assumption \ref{assBounded});  and inverse Lipschitz continuity of the
residual (Lemma \ref{lem:resinvlipschitz}) to derive the following
inequalities:
\begin{align}
\begin{split}
\underset{\solDummy\in
\spacetimesubspace
}{\min} \|\sol-\solDummy\|_2  &= \|\sol-\soltwo\|_2
\geq\frac{1}{\lipschitzK}
\|\res(\cdot;\soltwo)\|_{\weightmatst^T\weightmatst}\geq
\min_{\solDummy\in
\spacetimesubspace}\frac{1}{\lipschitzK}
\|\res(\cdot;\solDummy)\|_{\weightmatst^T\weightmatst}\\
&=
\frac{1}{\lipschitzK}
\|\res(\cdot;\solapproxST)\|_{\weightmatst^T\weightmatst}
\geq
\frac{\normEquiv}{\lipschitzK}
\|\res(\cdot;\solapproxST)\|_{2}
\geq
\frac{\normEquiv\inverseLipschitzK}{\lipschitzK}
\|\sol-\solapproxST\|_{2}.
\end{split}
\end{align}
Substituting in the definitions of $\lipschitzK$ and
$\inverseLipschitzK$ from
Eqs.~\eqref{eq:deflipschitzK} and \eqref{eq:definvlipschitzK}, respectively,
yields the stated result.
\end{proof}

\begin{theorem}[\textit{a priori} error bound with respect to $\ell^\infty$-optimal solution]\label{thm:apriori}
Under Assumptions \ref{assLipschitz}, \ref{asssmallDt},
and \ref{assBounded}, the error in the
\methodAcronym\ solution at any time instance can be bounded by the best
approximation error as
\begin{align}
\begin{split}
\max_{n\innat{\ntimedof}}\|\sol(t^n)-\solapproxST(t^n)\|_{2}\leq
\frac{\sqrt{\ntimedof}}{\normEquiv}
\left(
\frac{\sigmamax{\weightmatst\Alm} + \dt\lipschitzFK\sigmamax{\weightmatst\Blm} }
{\sigmamin{\Alm}-\dt\lipschitzFK\sigmamax{\Blm}}
\right)
\underset{\solDummy\in
\spacetimesubspace
}{\min} \max_{n\innat{\ntimedof}}\|\sol(t^n)-\solDummy(t^n)\|_2.
\end{split}
\end{align}
\end{theorem}
\begin{proof}
We begin by defining
the $\ell^\infty$-optimal solution
$\solinfty$, which satisfies
\begin{align}
\label{eq:optInfty}
&\solinfty = \underset{\solDummy\in
\spacetimesubspace
}{\arg\min} \|\sol-\solDummy\|_\infty,
\end{align}
where we have defined the $\ell^\infty$-norm as $\|\solDummy\|_\infty \defeq
\max_{n\innat{\ntimedof}}\|\solDummyFuncArg{n}\|_2$. Then, we can exploit
norm equivalence
$\|\solDummy\|_\infty\leq\|\solDummy\|_2\leq\sqrt{n}\|\solDummy\|_\infty$ for
$\solDummy\in\RR{n}$;  Lipschitz continuity (Lemma \ref{lem:reslipschitz}); residual-norm equivalence (Assumption \ref{assBounded});  and inverse Lipschitz continuity
of the residual (Lemma \ref{lem:resinvlipschitz}); and the
optimality properties of $\solinfty$ and $\solapproxST$ from Eqs.~\eqref{eq:optInfty} and
\eqref{eq:optLSPGST}, respectively, to derive the following inequalities:
\begin{align}
\begin{split}
\underset{\solDummy\in
\spacetimesubspace
}{\min} \|\sol-\solDummy\|_\infty  &= \|\sol-\solinfty\|_\infty
\geq\frac{1}{\sqrt{\ntimedof}}\|\sol-\solinfty\|_2
\geq\frac{1}{\lipschitzK\sqrt{\ntimedof}}
\|\res(\cdot;\solinfty)\|_{\weightmatst^T\weightmatst}\geq
\min_{\solDummy\in
\spacetimesubspace}\frac{1}{\lipschitzK\sqrt{\ntimedof}}
\|\res(\cdot;\solDummy)\|_{\weightmatst^T\weightmatst}\\
&=
\frac{1}{\lipschitzK\sqrt{\ntimedof}}
\|\res(\cdot;\solapproxST)\|_{\weightmatst^T\weightmatst}
\geq
\frac{\normEquiv}{\lipschitzK\sqrt{\ntimedof}}
\|\res(\cdot;\solapproxST)\|_{2}
\geq
\frac{\normEquiv\inverseLipschitzK}{\lipschitzK\sqrt{\ntimedof}}
\|\sol-\solapproxST\|_{2}\geq 
\frac{\normEquiv\inverseLipschitzK}{\lipschitzK\sqrt{\ntimedof}}
\|\sol-\solapproxST\|_{\infty}.
\end{split}
\end{align}
Noting that $\|\sol-\solapproxST\|_{\infty}\geq
\|\sol(t^n)-\solapproxST(t^n)\|_{2} $, $\forall n\innat{\ntimedof}$ and
substituting in the definitions of $\lipschitzK$ and
$\inverseLipschitzK$ from
Eqs.~\eqref{eq:deflipschitzK} and \eqref{eq:definvlipschitzK}, respectively,
yields the stated result.
\end{proof}

We now provide simplified variants of these error bounds in the case of
unweighted LSPG (Section \ref{sec:unweightedLSPG}) for which
$\weightmatstArg{n}=\identity{\ndof}$.

\begin{corollary}[Simplified \textit{a priori} error bound]\label{thm:simplifiedapriori}
If $\weightmatst = \identity{\nspacedof\ntimedof}$, then under Assumptions \ref{assLipschitz} and \ref{asssmallDt},
the error in the
\methodAcronym\ solution at any time instance can be bounded by the best
approximation error as
\begin{align}
\begin{split}
\label{eq:simplifiedAprioriL2}\|\sol-\solapproxST\|_{2}\leq
\left(1+
\lebesgue
\right)
\underset{\solDummy\in
\spacetimesubspace
}{\min}\|\sol-\solDummy\|_2,
\end{split}
\end{align}
\begin{align}\label{eq:simplifiedAprioriLinf}
\begin{split}
\max_{n\innat{\ntimedof}}\|\sol(t^n)-\solapproxST(t^n)\|_{2}\leq&
\sqrt{\ntimedof}
\left(1+
\lebesgue
\right)
\underset{\solDummy\in
\spacetimesubspace
}{\min} \max_{n\innat{\ntimedof}}\|\sol(t^n)-\solDummy(t^n)\|_2.
\end{split}
\end{align}
where we define the Lebesgue constant for a given time integrator and time step $\dt$ as
 \begin{equation} 
\lebesgue \defeq 
\frac{\sigmamax{\Alm}-\sigmamin{\Alm} + 2\dt\lipschitzFK\sigmamax{\Blm} }
{\sigmamin{\Alm}-\dt\lipschitzFK\sigmamax{\Blm}}.
 \end{equation} 
\end{corollary}
\begin{proof}
Proofs follows trivially from Theorems \ref{thm:aprioritwo}  and
\ref{thm:apriori} by substituting $\weightmatst =
\identity{\nspacedof\ntimedof}$ and noting 
that Assumption \ref{assBounded} is automatically satisfied for this choice of
weighting matrix $\weightmatst$, as $\normEquiv = 1$ in this case.
\end{proof}
\begin{remark}[Stability-constant growth]\label{rem:stabilityConstant}
Figure \ref{fig:stabilityConstants} plots the dependence of the stability
constants in the \textit{a priori} error bounds \eqref{eq:simplifiedAprioriL2}
and
\eqref{eq:simplifiedAprioriLinf} as a function of the final time $\totaltime$
for multiple linear multistep methods and fixed values of the time step and
	Lipschitz constant.
Critically, note that the stability constant for the $\ell^2$-norm of the
error in the \methodAcronym\ ROM solution grows only linearly in time, while
the stability constant for the $\ell^\infty$-norm of the
error in the \methodAcronym\ ROM solution exhibits polynomial growth in time
with degree $3/2$. Further, this trend is valid for all assessed linear multistep
schemes. This highlights one important feature of the proposed method:
significantly slower time growth of the solution error in time relative to
nonlinear model-reduction methods that perform only spatial projection, as
such error bounds grow exponentially in time \cite{rathinam:newlook,knezevic2011reduced,nguyen2009reduced,carlbergJCP}.
This is similar to the slow time growth of the error bounds demonstrated in
the context of the space--time reduced-basis
method \cite{urban2012new,urban2014improved,yanoPateraUrban,yanoReview}.
\end{remark}
\begin{figure}[htbp] 
\centering 
\subfigure[stability constant in inequality
\eqref{eq:simplifiedAprioriL2}]{
\includegraphics[height=0.45\textwidth]{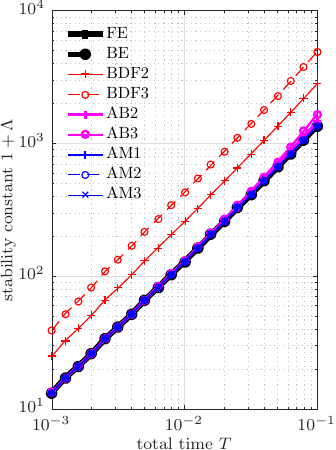} 
}~~~~~~~~~~~~~
\subfigure[stability constant in inequality
\eqref{eq:simplifiedAprioriLinf}]{
\includegraphics[height=0.45\textwidth]{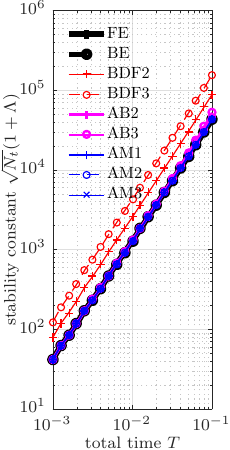} 
}
\caption{Stability constants in Corollary \ref{thm:simplifiedapriori} for
time step $\dt = 1\times 10^{-4}$, Lipschitz constant $\lipschitzFK = 1$, and
the following linear multistep methods: backward Euler (BE); backward
differentiation formulas (BDF2, BDF3); Adams--Bashforth with $s=2$ (AB2) and
$s=3$ (AB3); Adams--Moulton with $s=1$ (AM1), $s=2$ (AM2), and $s=3$ (AM3).
Note that the stability constant $1 + \lebesgue$ for the $\ell^2$-norm error
in inequality \eqref{eq:simplifiedAprioriL2} grows linearly in time, while the
stability constant 
$\sqrt{\ntimedof}(1 + \lebesgue)$ for the $\ell^\infty$-norm error
in inequality \eqref{eq:simplifiedAprioriLinf} exhibits polynomial time growth
with degree $3/2$.}
\label{fig:stabilityConstants} 
\end{figure} 

We now provide computable \textit{a posteriori} residual-based error bounds
and show that the \methodAcronym\ solution minimizes this bound over all
solutions in the space--time trial subspace.
\begin{corollary}[\textit{a posteriori} error
bound]\label{eq:aposterioribound}
Under Assumptions \ref{assLipschitz}, \ref{asssmallDt},
and \ref{assBounded}, the error in the
any approximation $\solDummy\in\spacetimesubspace$ can be bounded by the computed residual norm as
\begin{align}
\max_{n\innat{\ntimedof}}\|\sol(t^n)-\solDummy(t^n)\|_{2}\leq
\|\sol-\solDummy\|_{2}\leq
\frac{1}{\normEquiv\inverseLipschitzK}\|\res(\cdot;\solDummy)\|_{\weightmatst^T\weightmatst}.
\end{align}
Further, the \methodAcronym\ solution is the particular solution for which this
error bound is minimized, i.e., 
\begin{align}
\max_{n\innat{\ntimedof}}\|\sol(t^n)-\solapproxST(t^n)\|_{2}\leq
\|\sol-\solapproxST\|_{2}\leq
\frac{1}{\normEquiv\inverseLipschitzK}\min_{\solDummy\in
\spacetimesubspace}
\|\res(\cdot;\solDummy)\|_{\weightmatst^T\weightmatst}.
\end{align}
\end{corollary}
\begin{proof}
By invoking Assumption \ref{assBounded}, Lemma \ref{lem:resinvlipschitz}, and norm equivalence $\|\wbold\|_2\geq\|\wbold\|_\infty$, we can derive
\begin{align}
\begin{split}
\|\res(\cdot;\solapproxST)\|_{\weightmatst^T\weightmatst}
\geq
\normEquiv
\|\res(\cdot;\solapproxST)\|_{2}
\geq
\normEquiv\inverseLipschitzK
\|\sol-\solapproxST\|_{2}\geq 
\normEquiv\inverseLipschitzK
\|\sol-\solapproxST\|_{\infty},
\end{split}
\end{align}
which yields the first desired result. The second result follows from applying
the optimality property of the \methodAcronym\ solution \eqref{eq:optLSPGST}.
\end{proof}

\section{Numerical experiments}\label{sec:experiments}
This section compares the performance of the following methods:
\begin{itemize} 
\item \textit{FOM}. This model corresponds to the full-order model, i.e., the
solution satisfying Eq.~\eqref{eq:LMMresidual}.
\item \textit{LSPG ROM}. This model corresponds to the unweighted LSPG
ROM, i.e., the solution that satisfies Eq.~\eqref{eq:lspggnat} with $\weightmat =
\identity{\nspacedof}$.
\item \textit{GNAT ROM}. This model corresponds to the GNAT ROM, i.e., the
solution that satisfies Eq.~\eqref{eq:lspggnat} with $\weightmat =
(\samplemat \basismatres)^{+}\samplemat$. Algorithm 5 in Ref.~\cite{CarlbergGappy}
is used to construct the sampling matrix $\samplemat$.
\item \textit{ST-LSPG-1 ROM}. This model corresponds to the unweighted
\methodAcronym\ ROM,
i.e., the solution 
that satisfies Eq.~\eqref{eq:t-lspgReducedLargerFull} with
$\weightmatst = \identity{\nspacedof\ntimedof}$. 
The method is also characterized by the following:
 \begin{itemize} 
 \item Tailored temporal state subspaces computed according to
Eqs.~\eqref{eq:tailoredSVD}--\eqref{eq:tailoredSVD2}. 
\item As described in Section \ref{sec:initialGuess}, interpolation to
compute the initial guess. For this, we employ interpolation using linear radial
basis functions as described in Ref.~\cite{chirokov2016radial}.
\end{itemize}
\item \textit{ST-LSPG-2 ROM}.
This model is identical to the ST-LSPG-1 ROM except that it employs 
fixed temporal subspaces computed according to
Eqs.~\eqref{eq:temporalSVD3}--\eqref{eq:temporalSVD3Next}. 
\item \textit{ST-GNAT-1 ROM}. This model corresponds to the ST-GNAT ROM,
i.e., the solution 
that satisfies Eq.~\eqref{eq:t-lspgReducedLargerFull} with
$\weightmatst =
 (\samplematst\basismatrest)^+\samplematst$. 
 Otherwise, it is identical to the ST-LSPG-1 ROM with the additional following
 attributes:
 \begin{itemize} 
 \item Tailored temporal residual subspaces computed 
 according to 
 Eqs.~\eqref{eq:tailoredresSVD}--\eqref{eq:tailoredresSVD2}.
  \item Method \ref{resTrain:ROMtraining} in Section \ref{sec:spacetimeresidual} to
generate space--time residual samples, where
$\paramDomainTrainres=\paramDomainTrain$.
 \item Method \ref{greedyTimeThenSpace} in Section
 \ref{sec:constructSamplingMatrix} to construct the sampling matrix.
	 \end{itemize}
\item \textit{ST-GNAT-2 ROM}.
This model is identical to the ST-GNAT-1 ROM except that it employs 
a fixed temporal state subspace computed according to
Eqs.~\eqref{eq:temporalSVD3}--\eqref{eq:temporalSVD3Next}.
\end{itemize}

We assess the accuracy of any ROM solution $\solapproxFuncOnlyParam$ from its
mean squared state-space error, i.e.,
\begin{equation} 
\mseError  = \left. \sqrt{\sum_{n=1}^\ntimedof\|\solapproxFuncArg{n}-\solFuncArg{n}\|_2^2} \middle/ 
\sqrt{\sum_{n=1}^\ntimedof\|\solFuncArg{n}\|_2^2} \right., 
\end{equation} 
and we measure its computational cost in terms of the wall time incurred by the
ROM relative to that incurred by the FOM; the speedup is the
reciprocal of the relative wall time.  All timings are
obtained by performing calculations on an Intel(R) Xeon(R) CPU E5-2670 @ 2.60
GHz, 31.4 GB RAM using the \verb+MORTestbed+ \cite{zahr2010comparison} in
MATLAB. All reported timings are averaged over five simulations.

\subsection{Parameterized Burgers' equation}\label{sec:paramBurgers}
We first consider the parameterized inviscid Burgers' equation described in
Ref.~\cite{tpwl}, which corresponds to the following initial boundary value
problem for $x\in[0,1]$ and $t\in[0,\totaltime]$ with $\totaltime = 0.5$:
 \begin{align}\label{eq:burgers_eq} 
\begin{split}
   \frac{\partial w(x,t;\mu)}{\partial t} + \frac{\partial f(w(x,t;\mu))}{\partial x} &= 0.02e^{\mu_2 x}, 
   \quad \forall x\in[0,1], \quad \forall t\in[0,\totaltime]\\
   w(0,t;\param) &= \mu_1, \quad \forall t\in[0,\totaltime] \\
   w(x,0) &= 1, \quad \forall x\in[0,1]
\end{split}
 \end{align} 
 where $w: [0,1] \times [0,\totaltime]\times\paramDomain \rightarrow
 \RR{}$ is a conserved quantity and the $\nparam=2$ parameters comprise
 the left boundary value and source-term coefficient with 
$\param\equiv(\mu_1,\mu_2)\in\paramDomain = [1.2,1.5]\times[0.02,0.025]$.

After applying Godunov's scheme for spatial discretization with 100 control
volumes, Eqs.~\eqref{eq:burgers_eq} leads to a parameterized initial-value ODE
problem consistent with Eq.~\eqref{eq:fom} with $\nspacedof=100$ spatial
degrees of freedom. For time discretization, we employ the backward Euler scheme, which is a linear multistep method
characterized by $\kn=1$, $\alpha_0^n=\beta_0^n=1$, $\alpha_1^n=-1$,
$\beta_1^n=0$, $n\innat{\ntimedof}$. We \reviewerA{employ a uniform time step
of} $\dt = 2.5\times
10^{-4}$, leading to $\ntimedof = 2000$ time instances.  
% Figure \ref{fig:burgerFOM} shows the full-order model response for this problem for several
% chosen parameter instances.
% \begin{figure}[htbp] 
% \centering 
% 	\subfigure[$\param = (1.2, 0.02)$]{
%    \includegraphics[width=0.3\textwidth]{solutionFOM_burgersTrain1.eps}
% 	}
% 	\subfigure[$\param = (1.5, 0.02)$]{
%    \includegraphics[width=0.3\textwidth]{solutionFOM_burgersTrain4.eps}
% 	}
% 	\subfigure[$\param = (1.5,0.025)$]{
%    \includegraphics[width=0.3\textwidth]{solutionFOM_burgersTrain8.eps}
% 	}
% \caption{\textit{Burgers' equation.} FOM solutions for different parameter
% instances.}
%  \label{fig:burgerFOM} 
% 	\end{figure} 
For this problem, all ROMs employ a training set 
%$\paramDomainTrain= \{ (1.2,0.02),(1.3,0.02),(1.4,0.02),(1.5,0.02),(1.2,0.025),(1.3,0.025),(1.4,0.025),(1.5,0.025) \}$
$\paramDomainTrain= \{
	1.2,1.3,1.4,1.5
	\}\times\{0.02,0.025\}$ such that $\ntrain=8$
at which the FOM is solved. 

\reviewerBRtwo{We
	emphasize that (1) all assessed models employ the same time discretization;
	this includes the ST-LSPG and ST-GNAT
	ROMs, as the space--time residual is defined from this time discretization, and (2) we do not consider adaptive time-step
selection.
Future work will investigate
the effect of different time integrators---including those that employ
adaptive time-step selection---on the relative performance of
the methods.}
\subsubsection{Space--time bases}

 Figure \ref{fig:spacetime_basis_burgers} plots a selection of spatial and temporal
 modes computed using the three different techniques proposed in Section
 \ref{sec:spacetimetrialconstruct}. Note that the fixed temporal modes are nearly
 identical, regardless of whether the T-HOSVD or ST-HOSVD is employed. Thus,
 because the ST-HOSVD is significantly less computationally expensive, we no
 longer consider the fixed modes computed with T-HOSVD, which was the approach
 considered in Ref.~\cite{baumann2016space}.  On the other hand,
 the tailored temporal modes are significantly different from the fixed temporal
 modes. Further, they appear to be well suited for their respective spatial
 modes, as the temporal bases for higher-index spatial POD modes exhibit higher
 frequencies, which is consistent with previous studies (e.g.,
 Ref.~\cite{carlbergGalDiscOpt}).

\begin{figure}[htbp] 
  \centering 
  \subfigure[Spatial modes]{
  \includegraphics[width=0.3\textwidth]{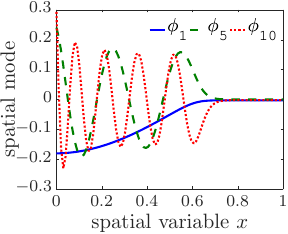} 
  }
  \subfigure[Fixed temporal modes, T-HOSVD]{
  \includegraphics[width=0.3\textwidth]{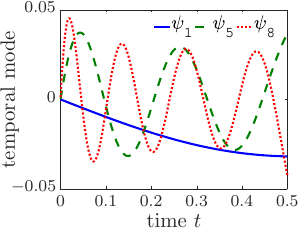} 
  }
  \subfigure[Fixed temporal modes, ST-HOSVD]{
  \includegraphics[width=0.3\textwidth]{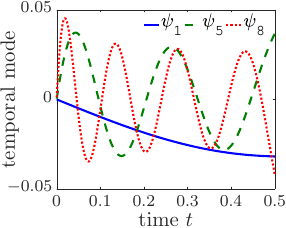} 
  }
  \subfigure[Tailored temporal modes for spatial mode $\basisvecspace_1$]{
  \includegraphics[width=0.3\textwidth]{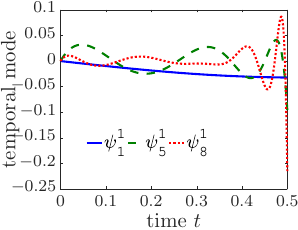} 
  }
  \subfigure[Tailored temporal modes for spatial mode $\basisvecspace_5$]{
  \includegraphics[width=0.3\textwidth]{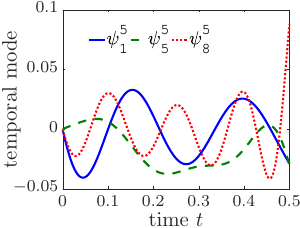} 
  }
  \subfigure[Tailored temporal modes for spatial mode $\basisvecspace_{10}$]{
  \includegraphics[width=0.3\textwidth]{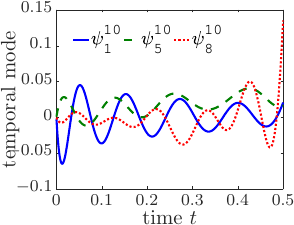} 
  }
  \caption{\textit{Burgers' equation.} Spatial and temporal modes computed
  using different tensor-decomposition techniques (see Section \ref{sec:spacetimetrialconstruct}).}
  \label{fig:spacetime_basis_burgers}
\end{figure}

\subsubsection{Model predictions}\label{sec:burgers_model_prediction}

We now compare the methods for fixed values of their parameters, and for
two randomly selected online points $\param^1 =
(1.35,0.0229)\not\in\paramDomainTrain$ and $\param^2 =
(1.45,0.0201)\not\in\paramDomainTrain$.  Table \ref{tab:burgersPerform}
reports the method parameter values and the associated performance of the
methods.  Figure \ref{fig:spacetime_solutions_burgers} reports snapshots of
the methods' responses for $t \in\{ 0,0.1665,0.3332,0.5\}$.  

\begin{table}[ht] 
\centering 
\small
 \begin{tabular}{|c||c|c||c|c||c|c|} 
\hline
method & LSPG & GNAT & ST-LSPG-1 & ST-LSPG-2 & ST-GNAT-1 & ST-GNAT-2\\
\hline
$\nbasisspace$ & 15 & 15 & 15 & 15 & 15 & 15 \\
$\nressample$  &    & 55 &    &    &    &    \\ 
$\nbasisres$   &    & 55 &    &    &    &    \\ 
$\nbasistime^i$&    &    & 2  &    & 2   &    \\ 
$\nbasistime$  &    &    &    & 20   &    &  20  \\ 
$\nresspaceind$&    &    &    &    & 30  & 30  \\ 
$\nrestimeind$ &    &    &    &    &120   & 120  \\ 
$\nbasisspaceres$&    &  &    &    & 100  & 100   \\ 
$\nbasistimeres^i$&    &  &    &    &  3  & 10  \\ 
 spatiotemporal dimension&  $3\times 10^4$  & $3\times 10^4$ & 30  & 300  & 30  & 300 \\ 
  \hline 
relative error for $\param^1$  & 0.00074& 0.011 & 0.0025 & 0.0011 & 0.0058 & 0.0063 \\
speedup   for $\param^1$     & 0.82  & 0.34   & 0.34   & 0.079  & 7.19   & 1.78 \\
\hline
relative error for $\param^2$& 0.0012 & 0.017 & 0.0038 & 0.0040 & 0.0077 & 0.0082 \\
speedup        for $\param^2$& 0.80   & 0.39  & 0.33   & 0.080  & 6.22   & 1.72 \\
\hline 
\end{tabular} 
\caption{\textit{Burgers' equation}. ROM method performance
  for fixed method parameters at randomly selected online points $\param^1 =
		(1.35,0.0229)\not\in\paramDomainTrain$ and $\param^2 =
		(1.45,0.0201)\not\in\paramDomainTrain$.} 
\label{tab:burgersPerform} 
\end{table} 
 \begin{figure}[htbp]
  \centering
	\subfigure[$\param = (1.35,0.0229)\not\in\paramDomainTrain$]{
   \includegraphics[width=0.45\textwidth]{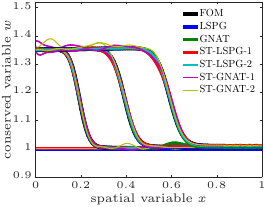} 
	}
	\subfigure[$\param = (1.45,0.0201)\not\in\paramDomainTrain$]{
   \includegraphics[width=0.45\textwidth]{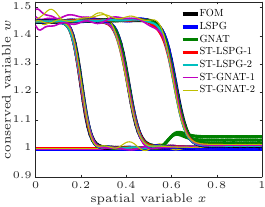} 
	}
  \caption{\textit{Burgers' equation.} 
	Method solutions for $t \in\{ 0,0.1665,0.3332,0.5\}$ corresponding to  method
parameters  reported in Table \ref{tab:burgersPerform}.
           }
  \label{fig:spacetime_solutions_burgers}
 \end{figure}

First, note that all ROMs generate accurate responses for this
particular combination of parameters, as the relative errors
are less than 1\% in all cases. Second, note that LSPG generates the most
accurate responses, but fails to generate any speedup due to its lack of
hyper-reduction. GNAT also fails to generate speedup in this case due to the
relatively small spatial dimension of the FOM and the larger number of Newton
iterations required for convergence relative to LSPG. The proposed
ST-LSPG methods incur slightly larger errors than the LSPG method, but they do
so with orders of magnitude fewer space--time degrees of freedom. This
highlights the promise of performing projection in both space and time: the
dimensionality of the problem can be significantly reduced while retaining
high levels of accuracy. However, due to their lack of hyper-reduction, the ST-LSPG
methods do not generate speedups. Finally, by employing
hyper-reduction, the ST-GNAT methods generate very
accurate predictions with significant speedups. 
We note that ST-LSPG-1 and ST-GNAT-1 exhibit better overall performance than ST-LSPG-2 and
ST-GNAT-2, respectively; this suggests that employing tailored temporal subspaces 
enables similar accuracy to be achieved using far fewer degrees of freedom, as
each temporal basis vector is tailored to its associated spatial basis vector.

\subsubsection{Method-parameter study}\label{sec:methodParam}
This section compares the performance of the ROM methods across a variation of
all method parameters. This study is essential to objectively compare the methods,
as the particular method-parameter values employed in Section
\ref{sec:burgers_model_prediction} did not necessarily yield optimal
performance for a given method.  For this reason, we subject each model to a
parameter study wherein each model parameter is varied between specified
limits; Table \ref{tab:burgersPerformParam} reports the tested parameter values for each method. 
We consider all elements in the resulting set if they satisfy the following
constraints:
$1.5\nbasisspace \leq \nbasisresst \leq \nressample$ for GNAT and
$1.5\nbasisst \leq \nbasisresst \leq \nresspaceind\nrestimeind$ for ST-GNAT-1
and ST-GNAT-2.
From these results, we then construct a Pareto front for each method,
which is characterized by the method parameters that minimize the competing
objectives of relative error and relative wall time. 
\begin{table}[ht] 
\centering 
\small
 \begin{tabular}{|c||c|c||c|c||c|c|} 
\hline
method & LSPG & GNAT & ST-LSPG-1 & ST-LSPG-2 & ST-GNAT-1 & ST-GNAT-2\\
\hline
$\nbasisspace$ & $\{10\times i\}_{i=1}^5$ &  $\{10\times i\}_{i=1}^5$ &  $\{10\times i\}_{i=1}^5$ &  $\{10\times i\}_{i=1}^5$ &  $\{10\times i\}_{i=1}^5$ & $\{10\times i\}_{i=1}^5$ \\
$\nressample$  &    & $\{20,30,40,60,80,90\}$&    &    &    &    \\ 
$\nbasisres$   &    & $\{20,30,40,60,80,90\}$ &    &    &    &    \\ 
$\nbasistime^i$&    &    & $\{3,4,5,6,8\}$  &    & $\{3,4,5,6,8\}$   &    \\ 
$\nbasistime$  &    &    &    & $\{5,10,20,30\}$   &    &  $\{5,10,20,30\}$  \\ 
$\nresspaceind$&    &    &    &    & $\{30,40\}$  & $\{30,40,60,70,80\}$  \\ 
$\nrestimeind$ &    &    &    &    &$\{60,120\}$   & $\{120\}$ \\ 
$\nbasisspaceres$&    &  &    &    &  $\{100\}$  & $\{100\}$   \\ 
$\nbasistimeres^i$&    &  &    &    &  $\{3\}$  & $\{10\}$  \\ 
 %trial-subspace dimension&  $2\times 10^4$  & $2\times 10^4$ & 20  & 200  & 20  & 200 \\ 
\hline 
\end{tabular} 
\caption{\textit{Burgers' equation}. Parameters used for the method-parameter
study. The set of tested parameters comprises the Cartesian product of the
specified parameter sets.
We consider all elements in the resulting set if they satisfy the following
guidelines:
$1.5\nbasisspace \leq \nbasisresst \leq \nressample$ for GNAT and
$1.5\nbasisst \leq \nbasisresst \leq \nresspaceind\nrestimeind$ for ST-GNAT-1 and ST-GNAT-2.
} 
\label{tab:burgersPerformParam} 
\end{table}

 Figure \ref{fig:paretoFrontBurgers} reports these Pareto fronts for the two
 online points, as well as an `overall' Pareto front that selects the
 Pareto-optimal methods across all parameter variations.  Table \ref{tab:burgersParamOverallPareto} reports
 values of the method parameters that yielded Pareto-optimal performance. The proposed
 ST-GNAT-1 method is Pareto-optimal for relative wall times less than one
 (i.e., faster than the FOM simulation).  While the proposed ST-GNAT-2 method
 does produce speedups, it is dominated by ST-GNAT-1; this provides further
 evidence of the advantage of employing a tailored relative to a fixed
 temporal basis.  We note that the worst-performing methods correspond to the
 ST-LSPG-1, and ST-LSPG-2 methods, as their lack of hyper-reduction leads to
 significant wall times that far exceed that of the FOM. Further,
 we note for a fixed error below a certain threshold, the ST-GNAT-1 method
 is nearly two orders of magnitude faster than the original GNAT method;
 this can be attributed to the fact that this approach reduces both the
 spatial and temporal complexities of the FOM. \reviewerA{Finally, we note
	 that because the spatial trial subspace employed by LSPG and GNAT has a
	 (relatively large) spatiotemporal dimension of $\nbasisspace\ntimedof$, while the space--time
	 trial subspace employed by ST-LSPG and ST-LSPG has a (relatively small) spatiotemporal
	 dimension of $\nbasisst(\ll\nbasisspace\ntimedof)$, the LSPG and GNAT
	 methods are able to generate smaller errors than the space--time methods.
	 However, this is achieved at significant computational cost that exceeds that
	 of the FOM in this case (i.e., relative wall times greater than one). Thus,
	 for this problem, LSPG is Pareto-optimal  \reviewerBRthree{and outperforms the space--time
	 ROMs} for relative errors less than
	 $10^{-6}$, although this regime is not useful because it incurs relative
 wall times greater than one.}

  \begin{figure}[htbp] 
    \centering 
		\subfigure[$\param^1 = (1.35,0.0229)\not\in\paramDomainTrain$]{
      \includegraphics[width=0.48\textwidth]{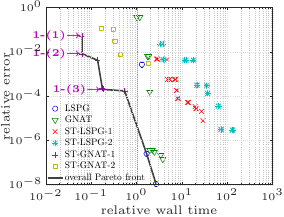} 
		}
		\subfigure[$\param^2 = (1.45,0.0201)\not\in\paramDomainTrain$]{
      \includegraphics[width=0.48\textwidth]{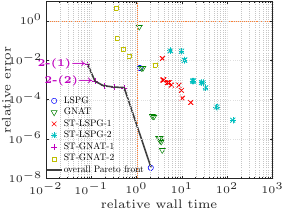} 
		}
    \caption{\textit{Burgers' equation.} Relative error versus relative wall time for
		varying model parameters reported in Table
		\ref{tab:burgersPerform}.}
    \label{fig:paretoFrontBurgers} 
  \end{figure} 

\begin{table}[ht] 
\centering 
\small
 \begin{tabular}{|c||c|c|c|c|c|c|} 
\hline
label & $\nbasisspace$ & $\nbasistime^i$ &  $\nresspaceind$ & 
$\nrestimeind$ & $\nbasisspaceres$ & $\nbasistimeres^i$ \\
\hline
 1-(1)&10&3&30&60&100&3 \\ 
 1-(2)&10&4&30&60&100&3 \\ 
 1-(3)&30&3&40&120&100&3 \\ \hline
 2-(1)&20&3&30&60&100&10 \\ 
 2-(2)&30&3&40&120&100&10 \\ 
 %trial-subspace dimension&  $2\times 10^4$  & $2\times 10^4$ & 20  & 200  & 20  & 200 \\ 
\hline 
\end{tabular} 
\caption{\textit{Burgers' equation}. Parameter values yielding Pareto-optimal
performance for the ST-GNAT-1 method. Figure \ref{fig:paretoFrontBurgers}
provides labels.} 
\label{tab:burgersParamOverallPareto} 
\end{table} 

\subsection{Quasi 1D Euler equation}
We now consider a parameterized quasi-1D Euler equation associated with
modeling inviscid compressible flow in a one-dimensional converging--diverging nozzle with a
continuously varying cross-sectional area \cite[Chapter 13]{maccormackNote};
Figure \ref{fig:converging-diverging_nozzle} depicts the problem geometry. 
The governing system of nonlinear PDEs is
 \begin{equation}\label{eq:euler} 
   \frac{\partial \solw}{\partial t} + \frac{1}{\crossareaSymb}\frac{\partial (\flux(\solw) \crossareaSymb)}{\partial x} 
   = \externalforce(\solw), \quad \forall x\in[0,1]\ \text{m}, \quad \forall t\in[0,\totaltime], 
 \end{equation} 
where $\totaltime=0.6$ s and
 \begin{equation} 
      \solw = \bmat{\densitySymb \\ \densitySymb \velocitySymb \\ \energydensity}, 
\quad \flux(\solw) = \bmat{\densitySymb \velocitySymb \\ \densitySymb \velocitySymb^2 + \pressureSymb \\ 
                    (\energydensity+\pressureSymb)\velocitySymb},
\quad \externalforce(\solw) = \bmat{0 \\ \frac{\pressureSymb}{\crossareaSymb}\frac{\partial \crossareaSymb}{\partial x}\\0},
\quad
      \pressureSymb = (\specificheat-1)\densitySymb\energypermass,
\quad \energypermass = \frac{\energydensity}{\densitySymb} - \frac{\velocitySymb^2}{2},
\quad \crossareaSymb = \crossareaSymb(x).
 \end{equation}
Here,
$\densitySymb$ denotes density,
$\velocitySymb$ denotes velocity,
$\pressureSymb$ denotes pressure, 
$\energypermass$ denotes potential energy per unit mass,
$\energydensity$ denotes total energy density,
$\specificheat$ denotes the specific heat ratio, and
$\crossareaSymb$ denotes the converging--diverging nozzle cross-sectional area.
We employ a
specific heat ratio of $\specificheat=1.3$, 
a specific gas constant of $\specificgasconstant=355.4$
$\text{m}^2/\text{s}^2/\text{K}$,
a total temperature of $\temperatureSymb_t = 300$ K,
and a total pressure of $\pressureSymb_t = 10^6$ $\text{N}/\text{m}^2$.
The cross-sectional area $A(x)$ is determined by a cubic spline interpolation over the points 
$(x,A(x)) \in \{ (0,0.2), (0.25,0.173), (0.5,0.17), (0.75, 0.173), (1,0.2)
\}$, which results in
 \begin{equation}
   A(x) = \left\{  
             \begin{array}{ll}
                 -0.288x^3 + 0.4080x^2 - 0.1920x + 0.2,                   &
								 x\in[0,0.25)\ \text{m} \\
           -0.288(x-0.25)^3 + 0.1920(x-0.25)^2 - 0.0420(x-0.25) + 0.1730, & x\in[0.25,0.5) \ \text{m}\\
            0.288(x-0.5)^3 - 0.0240(x-0.5)^2 + 0.17,                      & x\in[0.5,0.75) \ \text{m}\\
            0.288(x-0.75)^3 + 0.1920(x-0.75)^2 + 0.0420(x-0.75) + 0.1730, & x\in[0.75,1]\ \text{m}.
             \end{array}
          \right.   
 \end{equation} 
We assume a perfect gas that obeys the ideal gas law (i.e., $\pressureSymb = \densitySymb \specificgasconstant \temperatureSymb$).
The initial flow field is created in several steps.
First, the following isentropic relations are used to generate a zero pressure-gradient flow field:
 \begin{align}
   \begin{split}
     \machSymb(x) &= \frac{\machSymb_m\crossareaSymb_m}{\crossareaSymb(x)}\left ( \frac{1+\frac{\specificheat-1}{2}\machSymb(x)^2}{1+\frac{\gamma-1}{2}\machSymb_m^2} \right
  )^{\frac{\specificheat+1}{2(\specificheat-1)}},\quad
  \pressureSymb(x) = \pressureSymb_t\left (1+\frac{\specificheat-1}{2}\machSymb(x)^2\right )^{\frac{-\specificheat}{\specificheat-1}}
  \\ \temperatureSymb(x) &= \temperatureSymb_t\left
	(1+\frac{\specificheat-1}{2}\machSymb(x)^2\right )^{-1},\quad
  \densitySymb(x) = \frac{\pressureSymb(x)}{R\temperatureSymb(x)},\quad
  \speedofsoundSymb(x) =
	\sqrt{\specificheat\frac{\pressureSymb(x)}{\densitySymb(x)}} ,\quad
  \velocitySymb(x) = \machSymb(x) \speedofsoundSymb(x),
   \end{split}
 \end{align} 
 where a subscript $m$ indicates the flow quantity at $x=0.5$ m, and
 $\machSymb$ denotes the
 Mach number.
Then, a shock is placed at $x=0.85$ m of the flow field. 
   We use the jump relations for a stationary shock and 
   the perfect gas equation of state to derive the velocity across the shock
	 $\velocitySymb_2$, which satisfies
   the quadratic equation
 \begin{equation} \label{eq:quadratic}
   \left ( \frac{1}{2}-\frac{\specificheat}{\specificheat-1} \right ) \velocitySymb_2^2 
   + \frac{\specificheat}{\specificheat-1}\frac{n}{m}\velocitySymb_2 - h = 0.
 \end{equation}  
Here, $m \defeq \densitySymb_2\velocitySymb_2 = \densitySymb_1\velocitySymb_1$, 
      $n\defeq\densitySymb_2\velocitySymb_2^2+\pressureSymb_2 = \densitySymb_1\velocitySymb_1^2+\pressureSymb_1$,
      $h\defeq(\energydensity_2+\pressureSymb_2)/\densitySymb_2 =
			(\energydensity_1+\pressureSymb_1)/\densitySymb_1$, 
		and subscripts $1$ and $2$ denote a flow quantity to the left and to the
		right of
the shock, respectively. We employ the solution $\velocitySymb_2$
			to Eq.~\eqref{eq:quadratic}, which leads to a discontinuity (i.e., shock).
Finally, the exit pressure is increased to a factor $\pexit$ of its
original value in order to generate transient dynamics.  
\begin{figure}[htbp] 
  \centering 
  \includegraphics[width=0.4\textwidth]{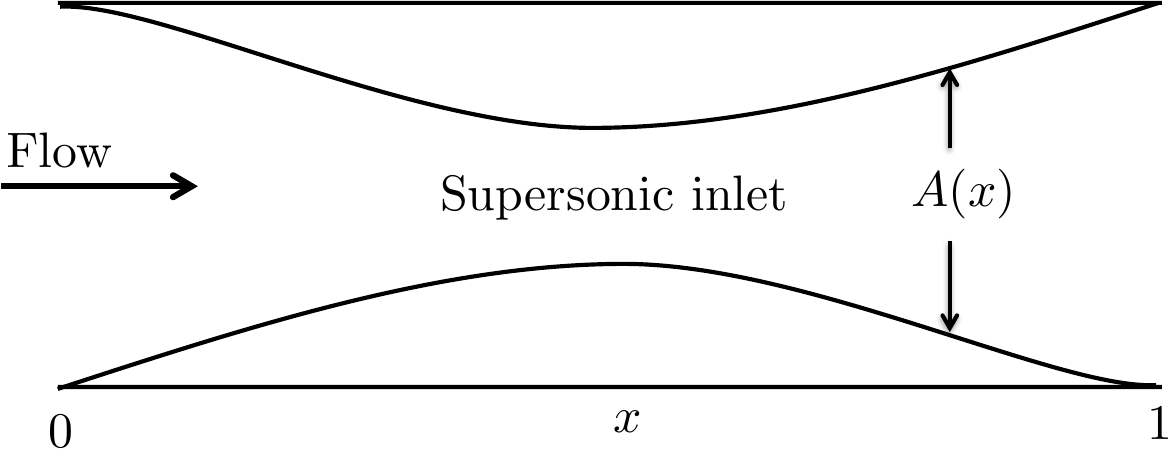} 
  \caption{\textit{Quasi-1D Euler.} Schematic figures of converging-diverging nozzle. } 
  \label{fig:converging-diverging_nozzle} 
\end{figure}

Applying a finite-volume spatial discretization with 50 equally spaced control
volumes and fully implicit boundary conditions leads to a parameterized system
of nonlinear ODEs consistent with Eq.~\eqref{eq:fom} with $\nspacedof = 150$
spatial degrees of freedom.  The Roe flux difference vector splitting method
is used to compute the flux at each intercell face \cite[Chapter
9]{maccormackNote}.  For time discretization, we again apply the backward
Euler scheme and a \reviewerA{uniform} time step of $\dt = 0.001$ s, leading to $\ntimedof=600$.

For this problem, we use the following two parameters: 
the pressure factor $\paramSymb_1 = \pexit$ and the Mach number at the middle of the nozzle
$\paramSymb_2 = \machSymb_m$.
All ROMs employ a training set at which the FOM is solved of
$\paramDomainTrain = \{1.7+0.01i\}_{i=0}^3\times \{1.7,1.72\}$ such that
$\ntrain = 8$.
% Figure \ref{fig:solutionFOM_euler} shows the full-order model response 
% for this problem for several chosen parameter instances.
% \begin{figure}[htbp] 
% \centering 
% 	\subfigure[$\param = (1.50,1.7)$]{
%    \includegraphics[width=0.3\textwidth]{solutionFOM_euler.eps}
% 	}
% 	\subfigure[$\param = (1.56,1.7)$]{
%    \includegraphics[width=0.3\textwidth]{solutionFOM_euler.eps}
% 	}
% 	\subfigure[$\param = (1.56,1.72)$]{
%    \includegraphics[width=0.3\textwidth]{solutionFOM_euler.eps}
% 	}
% \caption{\textit{Quasi-1D Euler equation.} Density ($\densitySymb$) component of the FOM solution for different parameter
% instances.} 
%  \label{fig:solutionFOM_euler} 
% 	\end{figure} 

\subsubsection{Space--time bases}

 Figure \ref{fig:spacetime_basis_euler} plots several spatial and temporal
 modes computed using the three different techniques proposed in Section
 \ref{sec:spacetimetrialconstruct}. As with the Burgers equation, the `fixed'
 temporal modes are nearly identical, regardless of whether the T-HOSVD or
 ST-HOSVD is employed, rendering the ST-HOSVD more appealing due to its
 reduced computational cost. In addition, the tailored temporal modes are
 significantly different, with the temporal basis exhibiting higher
 frequencies for higher-index spatial modes as expected.

\begin{figure}[htbp] 
  \centering 
  \subfigure[Spatial modes]{
  \includegraphics[width=4.7cm]{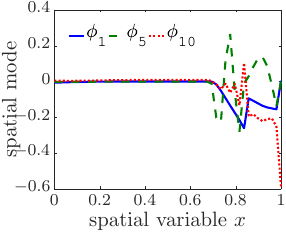} 
  }
  \subfigure[Fixed temporal modes, T-HOSVD]{
  \includegraphics[width=4.7cm]{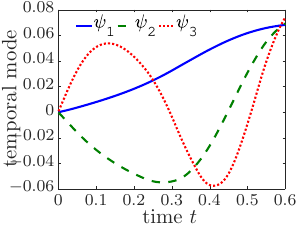} 
  }
  \subfigure[Fixed temporal modes, ST-HOSVD]{
  \includegraphics[width=4.7cm]{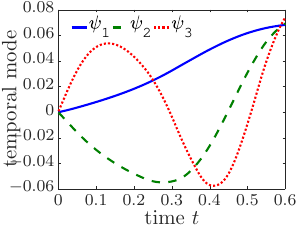} 
  }
  \subfigure[Tailored temporal modes for spatial mode $\basisvecspace_1$]{
  \includegraphics[width=4.7cm]{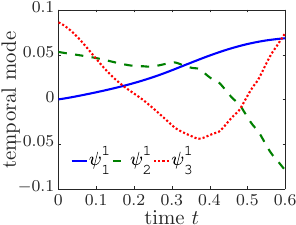} 
  }
  \subfigure[Tailored temporal modes for spatial mode $\basisvecspace_5$]{
  \includegraphics[width=4.7cm]{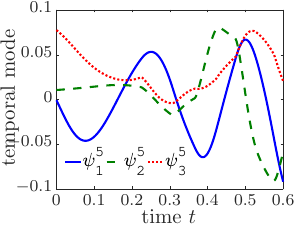} 
  }
  \subfigure[Tailored temporal modes for spatial mode $\basisvecspace_{10}$]{
  \includegraphics[width=4.7cm]{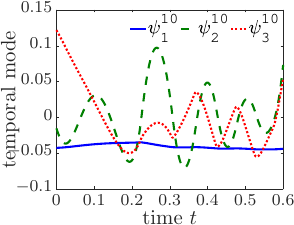} 
  }
  \caption{\textit{Quasi-1D Euler equation.} Spatial and temporal modes computed
  using different techniques (see Section \ref{sec:spacetimetrialconstruct}).}
  \label{fig:spacetime_basis_euler}
\end{figure}

\subsubsection{Model predictions}\label{sec:euler_model_prediction}

We now compare the methods for fixed values of their parameters, and for
two randomly selected online points $\param^1 =
		(1.7125,1.71)\not\in\paramDomainTrain$ and $\param^2 =
		(1.7225,1.705)\not\in\paramDomainTrain$.  Table \ref{tab:eulerPerform}
reports the method parameter values and the associated performance of the
methods.  Figure \ref{fig:spacetime_solutions_euler} reports snapshots of
the methods' responses for $t\in\{0,\totaltime\}$.

\begin{table}[ht] 
\centering 
\small
 \begin{tabular}{|c||c|c||c|c||c|c|} 
\hline
method & LSPG & GNAT & ST-LSPG-1 & ST-LSPG-2 & ST-GNAT-1 & ST-GNAT-2\\
\hline
$\nbasisspace$ & 50 & 50 & 50 & 50 & 50 & 50 \\
$\nressample$  &    & 145 &    &    &    &    \\ 
$\nbasisres$   &    & 145 &    &    &    &    \\ 
$\nbasistime^i$&    &    & 3  &    & 3   &    \\ 
$\nbasistime$  &    &    &    & 30   &    &  30  \\ 
$\nresspaceind$&    &    &    &    & 120  & 140  \\ 
$\nrestimeind$ &    &    &    &    &20   & 100  \\ 
$\nbasisspaceres$&    &  &    &    &  150  & 150   \\ 
$\nbasistimeres^i$&    &  &    &    &  10  & 10 \\ 
 spatiotemporal dimension&  $3\times 10^4$  & $3\times 10^4$ & 150  & $1.5\times 10^3$  &
 150  & $1.5\times 10^3$ \\ 
  \hline 
relative error for $\param^1$  & $7.78\times 10^{-6}$ & 0.55 & 0.012 & $6.3\times 10^{-4}$ & 0.0023 & 0.0048 \\
speedup   for $\param^1$     &0.77 & 1.04 & 0.84 & 0.58 & 21.79 & 0.49  \\
\hline
relative error for $\param^2$& $8.31\times 10^{-6}$ & 0.026 & 0.0076 & 0.0021 & 0.0025 & 0.0040 \\
speedup        for $\param^2$& 0.81 & 1.01 & 0.85 & 0.41 & 22.52 & 2.79  \\
\hline 
\end{tabular} 
\caption{\textit{Quasi-1D Euler equation}. ROM method performance
  for fixed method parameters at randomly selected online points $\param^1 =
		(1.7125,1.71)\not\in\paramDomainTrain$ and $\param^2 =
		(1.7225,1.705)\not\in\paramDomainTrain$.} 
\label{tab:eulerPerform} 
\end{table}

 \begin{figure}[htbp]
  \centering
	\subfigure[$\param = (1.7125,1.71)\not\in\paramDomainTrain$]{
   \includegraphics[width=0.45\textwidth]{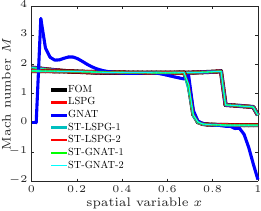} 
	}
	\subfigure[$\param = (1.7225,1.705)\not\in\paramDomainTrain$]{
   \includegraphics[width=0.45\textwidth]{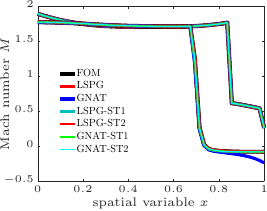} 
	}
  \caption{\textit{Quasi-1D Euler equation.} 
	Method solutions for $t \in\{ 0,\totaltime\}$ corresponding to  method
parameters  reported in Table \ref{tab:eulerPerform}.
}
  \label{fig:spacetime_solutions_euler}
 \end{figure}

Conclusions are similar to those derived from the Burgers' equation results.
First, note that all ROMs except for GNAT in the case of $\param^1$  generate accurate responses, as the relative errors
are less than 3\% in all cases. Second, as before, LSPG generates the most
accurate responses, but fails to generate any speedup due to its lack of
hyper-reduction. The proposed
ST-LSPG methods incur sub-2\% errors, but they do
so with orders of magnitude fewer spatiotemporal degrees of freedom relative to
the LSPG and GNAT methods, which 
highlights the promise of performing projection in both space and time. Again,
as these methods do not employ hyper-reduction, they do not generate speedups.
Finally, the ST-GNAT methods 
generate both 
accurate predictions with significant speedups.
We note that ST-GNAT-1 performs better than 
ST-GNAT-2, providing further evidence of the ability of tailored bases to produce
accurate responses with fewer degrees of freedom.

\subsubsection{Method-parameter study}\label{sec:methodParamtwo}

We again compare the performance of the ROM methods across a wide variation of
all method parameters.  Table \ref{tab:eulerPerformParam} reports the tested
parameter values for each method.  We consider all elements in the resulting
set if they satisfy constraints $1.5\nbasisspace \leq \nbasisresst \leq
\nressample$ for GNAT and $1.5\nbasisst \leq \nbasisresst \leq
\nresspaceind\nrestimeind$ for ST-GNAT-1 and ST-GNAT-2.  From these results,
we then construct a Pareto front for each method, which is characterized by
the method parameters that minimize the competing objectives of relative error
and relative wall time. 

\begin{table}[ht] 
\centering 
\small
 \begin{tabular}{|c||c|c||c|c||c|c|} 
\hline
method & LSPG & GNAT & ST-LSPG-1 & ST-LSPG-2 & ST-GNAT-1 & ST-GNAT-2\\
\hline
$\nbasisspace$ & $\{10\times i\}_{i=1}^6$ &  $\{10\times i\}_{i=1}^6$ &  $\{10\times i\}_{i=1}^6$ &  $\{10\times i\}_{i=1}^6$ &  $\{10\times i\}_{i=1}^6$ & $\{10\times i\}_{i=1}^6$ \\
$\nressample$  &    & $\{10\times i\}_{i=2}^{10}\cup\{120,145\}$&    &    &    &    \\ 
$\nbasisres$   &    & $\{10\times i\}_{i=2}^{10}\cup\{120,145\}$ &    &    &    &    \\ 
$\nbasistime^i$&    &    & $\{i\}_{i=3}^8$  &    & $\{i\}_{i=3}^8$   &    \\ 
$\nbasistime$  &    &    &    & $\{10\times i\}_{i=2}^5$   &    &  $\{10\times i\}_{i=2}^5$  \\ 
$\nresspaceind$&    &    &    &    & $\{120\}$  & $\{140\}$  \\ 
$\nrestimeind$ &    &    &    &    &$\{10,20,30,40,60\}$   & $\{30,50,100,150\}$ \\ 
$\nbasisspaceres$&    &  &    &    &  150  & 150   \\ 
$\nbasistimeres^i$&    &  &    &    &  10  & 10  \\ 
 %trial-subspace dimension&  $2\times 10^4$  & $2\times 10^4$ & 20  & 200  & 20  & 200 \\ 
\hline 
\end{tabular} 
\caption{\textit{Quasi-1D Euler equation}. Parameters used for the method-parameter
study. The set of tested parameters comprises the Cartesian product of the
specified parameter sets.
We consider all elements in the resulting set if they satisfy the following
guidelines:
$1.5\nbasisspace \leq \nbasisresst \leq \nressample$ for GNAT and
$1.5\nbasisst \leq \nbasisresst \leq \nresspaceind\nrestimeind$ for ST-GNAT-1 and ST-GNAT-2.
} 
\label{tab:eulerPerformParam} 
\end{table} 

Figure \ref{fig:paretoFrontEuler} reports these Pareto fronts for the two
online points, as well as an overall Pareto front that selects the
Pareto-optimal methods across all parameter variations.  
Table \ref{tab:eulerParamOverallPareto} reports
 values of the method parameters that yielded Pareto-optimal performance.
These results show
that---as before---the proposed ST-GNAT-1 method is Pareto optimal for
relative wall time less than 0.9 and relative errors less than 20\%.  While
the proposed ST-GNAT-2 method produces speedups, it is again dominated by
ST-GNAT-1, further highlighting the advantage of tailored versus fixed
temporal bases.  Again, the worst-performing methods correspond to the
ST-LSPG-1, and ST-LSPG-2 methods, as their lack of hyper-reduction leads to
significant wall times that far exceed that of the FOM. Further,
we note that for a fixed error below a certain threshold, the ST-GNAT-1 method is
over one order of magnitude faster than the original GNAT method; this can
be attributed to the fact that this approach reduces both the spatial and
temporal complexities of the FOM.
\reviewerA{Finally, we again note that
	LSPG is Pareto-optimal \reviewerBRthree{and outperforms the space--time
	ROMs} for extremely small \reviewerBRthree{relative} errors
	\reviewerBRthree{less than approximately $3\times 10^{-5}$} due to the higher
	spatiotemporal dimensionality of the spatial trial subspace; however, this
	regime is not useful for this problem, as it leads to LSPG models roughly as
	expensive as the FOM (i.e., relative wall times near one).}

  \begin{figure}[htbp] 
    \centering 
		\subfigure[$\param = (1.7125,1.71)\not\in\paramDomainTrain$]{
      \includegraphics[width=0.48\textwidth]{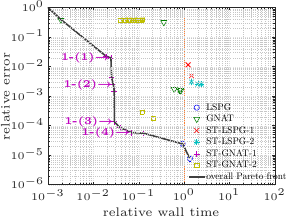} 
		}
		\subfigure[$\param = (1.7225,1.705)\not\in\paramDomainTrain$]{
      \includegraphics[width=0.48\textwidth]{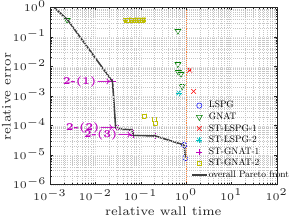} 
		}
    \caption{\textit{Quasi-1D Euler equation}. Relative error versus relative wall time for
		varying model parameters reported in Table \ref{tab:eulerPerformParam}.}
    \label{fig:paretoFrontEuler} 
  \end{figure} 

\begin{table}[ht] 
\centering 
\small
 \begin{tabular}{|c||c|c|c|c|c|c|} 
\hline
label & $\nbasisspace$ & $\nbasistime^i$ & $\nresspaceind$ & 
$\nrestimeind$ & $\nbasisspaceres$ & $\nbasistimeres^i$  \\
\hline
 1-(1)&10&3&120&10&150&10 \\ 
 1-(2)&60&3&120&10&150&10 \\ 
 1-(3)&30&4&120&10&150&10 \\ 
 1-(4)&50&8&120&10&150&10 \\ \hline
 2-(1)&30&3&120&10&150&10 \\ 
 2-(2)&50&4&120&10&150&10 \\ 
 2-(3)&60&7&120&10&150&10 \\ 
 %trial-subspace dimension&  $2\times 10^4$  & $2\times 10^4$ & 20  & 200  & 20  & 200 \\ 
\hline 
\end{tabular} 
\caption{\textit{Quasi-1D Euler equation}. 
Parameter values yielding Pareto-optimal
performance for the ST-GNAT-1 method. Figure \ref{fig:paretoFrontEuler}
provides labels.
} 
\label{tab:eulerParamOverallPareto} 
\end{table} 

\section{Conclusions}\label{sec:conclusions}

This work proposed a model-reduction method for nonlinear dynamical systems
based on space--time least-squares Petrov--Galerkin projection. The method
computes optimal approximations by minimizing the discrete space--time
residual over all elements in a low-dimensional space--time trial subspace in
a weighted $\ell^2$-norm. Advantages of the method include:
 \begin{itemize} 
\item 
its ability to
reduce both the spatial and temporal dimensions of the dynamical system
(Remark \ref{rem:stlspgreducedim}), 
%\item its \reviewerBRtwo{ability to remove} spurious temporal modes
%from the ROM response (Remark \ref{rem:spurious}),
\item \reviewerBRthree{\textit{a priori}} error bounds that
\reviewerBRthree{bound the solution error by the best space--time
approximation error and whose stability constants}
exhibit subquadratic growth in time (Section \ref{sec:error}),
\item applicability to general nonlinear dynamical systems,  
\item hyper-reduction
that reduces the complexity in the presence of general nonlinearities (Section
\ref{sec:spacetimehyper}), and
\item its ability to extract multiple space--time
basis vectors from each training simulation via tensor decomposition (Section
\ref{sec:error}).
	 \end{itemize}
In addition to introducing the novel \methodAcronym\ method, this work
proposed specific approaches to computing the method's ingredients: the
space--time trial subspace (Section \ref{sec:spacetimetrialconstruct}), the
space--time residual basis in the case of ST--GNAT (Section
\ref{sec:spacetimeresidual}), the sampling matrix in the case of
hyper-reduction (Section \ref{sec:constructSamplingMatrix}), and the initial
guess used in the Gauss--Newton method applied to solve the nonlinear
least-squares problem (Section \ref{sec:initialGuess}).  Numerical experiments
demonstrated the ability of the proposed method to generate
orders-of-magnitude speedups over existing spatial-projection-based ROMs
without sacrificing accuracy.  

Future work entails implementing the method in parallel
computational-mechanics codes\reviewerA{,} devising techniques to reduce the amount of
storage required for the state and residual tensors\reviewerBRtwo{, and 
	assessing
the effect of different time integrators---as well as adaptive
time-stepping---on method performance.
}

\section*{Acknowledgments}
The authors gratefully acknowledge Tamara Kolda and Grey Ballard for
insightful discussions that led to the tensor-decomposition approaches for computing the
space--time trial subspace.  The authors also acknowledge Professors Benjamin
Peherstorfer and Masayuki Yano for useful comments received at the Model
Reduction for Parametrized Systems (MoRePaS) III workshop and the 2017 SIAM
Conference on \reviewerA{Computational} Science and Engineering, respectively. 
\ourReReading{The authors also gratefully acknowledge the helpful comments
provided by the anonymous reviewers.} 
\ourReReading{This work was performed at Sandia National Laboratories and was
supported by the LDRD program (project 190968).}
\ourReReading{Sandia National Laboratories is a multimission laboratory
	managed and operated by National Technology and Engineering Solutions of
	Sandia, LLC., a wholly owned subsidiary of Honeywell International, Inc.,
	for the U.S.\ Department of Energy's National Nuclear Security Administration
	under contract DE-NA-0003525.}
\ourReReading{Lawrence Livermore National Laboratory is operated by Lawrence 
Livermore National Security, LLC, for the U.S. Department of Energy, National 
Nuclear Security Administration under Contract DE-AC52-07NA27344.}

\appendix

\bibliography{references}
\bibliographystyle{siam}
\end{document}